\documentclass[11pt]{article}
\usepackage{amsmath}
\usepackage{amssymb}
\usepackage{graphicx}
\usepackage{tikz}
\usepackage{xcolor}
\usepackage{amsthm}
\usepackage{amsmath,accents}
\usepackage{esint}
\usepackage{geometry}                
\usepackage{amsmath, amsthm, amssymb}
\usepackage{graphicx}
\usepackage{amssymb}
\usepackage{epstopdf}
\usepackage{pdfsync}

\title{Optimal design of mixtures of ferromagnetic interactions}
\author{Andrea Braides, \\ \small Dipartimento di Matematica, Universit\`a di Roma Tor Vergata
\\ \small via della ricerca scientifica 1, 00133 Roma, Italy \\\\
Leonard Kreutz, \\ \small
Gran Sasso Science Institute\\ \small
 Viale Francesco Crispi 7, 67100 L'Aquila,Italy}
\date{}                                           

\def\R{\mathbb R}
\def\e{\varepsilon}

\begin{document}
\maketitle
\newtheorem{lemma}{Lemma}
\newtheorem{corollary}[lemma]{Corollary}
\newtheorem{proposition}[lemma]{Proposition}
\newtheorem{remark}[lemma]{Remark}
\newtheorem{example}[lemma]{Example}
\newtheorem{theorem}[lemma]{Theorem}
\newtheorem{definition}[lemma]{Definition}
\begin{abstract} 
We provide a general framework for the optimal design of surface energies on networks.
We prove sharp bounds for the homogenization of discrete systems describing mixtures of ferromagnetic interactions by constructing optimal microgeometries, and we prove a localization principle which allows to reduce to the periodic setting in the general nonperiodic case.
\end{abstract}



\section{Introduction}
The optimization of the design of structures can sometimes be viewed as a minimization or maximization problem of a cost or compliance subjected to design constraints. A typical example
is shape optimization for given loads of conducting or elastic structures composed of a prescribed amounts of a certain number of materials. In that case the existence of an optimal shape is not guaranteed, and a relaxed formulation must be introduced that takes into account the possibility of fine mixtures.
The homogenization method as presented for example in the book by Allaire \cite{Allaire}
can be regarded as subdividing the problem into the description of all possible materials obtained as mixtures, and subsequently optimize in the enlarged class of homogenized materials that satisfy the corresponding relaxed design constraint.

In this paper we extend the homogenization method to the
optimal design of networks for surface energies. From a standpoint of Statistical Mechanics, the object of our study are mixtures of ferromagnetic interactions under the constraint that
interaction coefficients (bonds) may only take values in a fixed set of parameters.  
We consider energies of Ising type defined on a cubic lattice, of the form
\begin{equation}
E(u)=\frac{1}{4}\sum_{\xi \in V}\sum_{i \in \mathbb{Z}^d}c_{i,\xi}(u_i-u_{i+\xi})^2
\end{equation}
where $u_i\in \{-1,1\}$, $V \subset \mathbb{Z}^d$ is a finite set describing the range of interactions,
and the coefficients $c_{i,\xi}$, representing the strength of the interaction at the point $i$ in direction $\xi$, satisfy some design constraint described below. Note that we prefer to consider the interaction in the form $c_{i,\xi}(u_i-u_{i+\xi})^2$ rather that in the (equivalent, up to a scaling factor) form $-c_{i,\xi}u_iu_{i+\xi}$, which is more customary in Statistical Mechanics, since in this way the energy of ground states is normalized to $0$ and we avoid possible indeterminate forms in the case of infinite domains.
In order to describe the surface energy corresponding to such a system we follow a discrete-to-continuum approach as in \cite{BP,AS} (see also \cite{CdL,B}): we scale the
energies introducing a parameter $\e$ and defining 
\begin{equation}
E_\e(u)=\frac{1}{4}\sum_{\xi \in V}\sum_{i \in \mathbb{Z}^d}\e^{d-1}c_{i,\xi}(u_i-u_{i+\xi})^2,
\end{equation}
where now the functions $u$ are considered as defined on $\e\mathbb{Z}^d$, with $u_i$
the value at $\e i$. By identifying each $u$ with the corresponding piecewise-constant interpolation on $\e\mathbb{Z}^d$ such energies can be considered as defined in a Lebesgue space $L^1$, where they are equicoercive, so that their $\Gamma$-limit can be used as a continuum approximation in the description of the corresponding minimum problem \cite{GCB}.

Since in our optimal-design problem we have to take into account the possibility of locally varying the arrangement of the interactions we further introduce a dependence on $\e$ on
the coefficients, and consider energies 
\begin{equation}\label{etre}
E_\e(u)=\frac{1}{4}\sum_{\xi \in V}\sum_{i \in \mathbb{Z}^d}\e^{d-1}c^\e_{i,\xi}(u_i-u_{i+\xi})^2.
\end{equation}
Compactness theorems \cite{AS} ensure that the $\Gamma$-limit of such energies exists up to subsequences, is finite on functions in $BV_{\rm loc}(\R^d;\{-1,1\})$, and takes the form of a surface energy
\begin{equation}\label{equattro}
F(u)=\int_{\partial^*\{u=1\}}\varphi(x,\nu_u)d{\mathcal H}^{d-1},
\end{equation}
where $\partial^*\{u=1\}$ is the reduced boundary of $\{u=1\}$ and $\nu_u$ is its inner normal.

The optimal-design constraint that we consider is that for fixed $\xi\in V$ the bonds $c_{i,\xi}$ may take two positive values $\alpha_\xi$ and $\beta_\xi$ depending on $\xi$ with 
\begin{equation}
\alpha_\xi<\beta_\xi.
\end{equation}
The simplest case is that of nearest neighbours, when $V=\{e_1,\ldots,e_d\}$ is the canonical basis of
$\R^d$ and the strength of the bonds is independent of the direction: $\alpha_{e_j}=\alpha$, $\beta_{\e_j}=\beta$. In that case we are mixing two types of connections in a cubic lattice. A simplified  description of the two-dimensional setting for nearest neighbour interactions can be found in \cite{BK}.

The first step in the homogenization method is to consider all possible $\varphi$ in the periodic-bond setting; that is, when $i\mapsto c_{i,\xi}$ are periodic, in which case $\varphi$ is $x$-independent.
We denote such $\varphi$, extended to $\R^d$ by positive homogeneity of degree one, as the {\em homogenized surface tension} of the system $\{c_{i,\xi}\}$. The description of such $\varphi$ with fixed {\em volume fraction} (proportion) $\theta$ of $\beta$-type bonds is what is usually referred to as a {\em G-closure} problem, with a terminology borrowed from elliptic homogenization \cite{Allaire,Milton}.
We show that all possible such $\varphi$ are the (positively homogeneous of degree one) symmetric convex functions
such that
\begin{equation}\label{bobo}
\sum_{\xi \in V}\alpha_\xi |\langle \nu,\xi\rangle|\le \varphi(\nu)\le \sum_{\xi \in V} (\beta_\xi\theta_\xi +(1-\theta_\xi) \alpha_\xi) |\langle \nu, \xi\rangle|,
\end{equation}
where the $\theta_\xi \in [0,1]$ are the {\em volume fraction} of the interaction coefficients which account only for points $i$ interacting with points  $i+\xi$, and satisfy
\begin{equation*}
\frac{1}{\# V}\sum_{\xi \in V} \theta_\xi =\theta.
\end{equation*}
Note that, strictly speaking, such a description makes sense only for $\theta$ a rational number. We denote by ${\bf H}(\theta)$ the family of all $\varphi$ as above satisfying (\ref{bobo}). If $\theta$ is not a rational number, then the elements of ${\bf H}(\theta)$ are regarded as approximated by elements of ${\bf H}(\theta_h)$ with $\theta_h\to\theta$ and rational.

In dimension two we may compare this $G$-closure problem with a
continuous analog on curves in $\R^2$, which consists in the determination of optimal bounds
for Finsler metrics  obtained from the homogenization of periodic Riemannian
metrics (see \cite{AcBu,BDF,BBF}) of the form
$$
\int_a^b a\Bigl({u(t)\over\e}\Bigr)|u'|^2dt,
$$
and $a(u)$ is a periodic function in $\R^d$ taking only the values $\alpha$ and $\beta$. Even though curves and boundaries of sets have some topological differences, it has been shown in \cite{BP}, that in the periodic setting the homogenized energy densities can be computed by optimal paths (curves) on the dual lattice.
The problem on curves has been studied in \cite{DP}, where it is shown that homogenized metrics
satisfy
$$
\alpha|\nu|\le \varphi(\nu)\le (\theta\beta+(1-\theta)\alpha)|\nu|,
$$
but the optimality of such bounds is not proved. That result provides bounds also for the `dual' equivalent formulation in dimension $2$ of the homogenization of 
periodic perimeter functionals of the form
$$
\int_{\partial A}a\Bigl({x\over\e}\Bigr) \,d{\cal H}^{d-1}(x)
$$
with the same type of $a$ as above (see \cite{AB,AI}). The corresponding $\varphi$ in this case 
can be interpreted as the surface tension of the homogenized perimeter functional.

The discrete setting allows to give a (relatively) easy description of the optimal bounds in
a way similar to the treatment of mixtures of linearly elastic discrete structures \cite{BF}.
The bounds obtained by sections and by averages in the elastic case have as counterpart 
{\em bounds by projection}, where the homogenized surface tension is estimated from below by
considering the minimal value of the coefficient on each section, and {\em bounds by averaging},
where coefficients on a section are substituted with their average. The discrete setting allows to
construct (almost-)optimal periodic geometries, which optimize one type or the other of the bounds in every direction $\xi$  at the same time. Since the constructions of optimal geometries for fixed $\xi$ may overlap, some extra care must be used to make sure that they are compatible. This is done by a separation of scales argument. 

The homogenization method is completed by proving that we may always  locally reduce to 
the case of periodic coefficients. More precisely, we note that, up to subsequences, in the general non-periodic setting of energies as in ({\ref{etre}),
we may define a {\em continuum local volume fraction} $\theta=\theta(x)$ describing the local percentage of
$\beta$-type bonds, as the average of the densities $\theta_\xi(x)$ of the weak$^*$-limit of the measures
$$
\mu^\xi_\e= \sum_{i: c^\e_{i,\xi}=\beta_\xi} \delta_{\e i}.
$$
We then prove a {\em localization principle}, similar to the one for quadratic gradient energies in the Sobolev space setting stated by Dal Maso and Kohn (see \cite{R,B-handbook}). In our case,
this amounts to proving that all $\varphi$ that we may obtain in (\ref{equattro}) are exactly those such that, upon suitably choosing their representative, 
\begin{equation}\label{eotto}
\varphi(x,\cdot)\in {\bf H}(\theta(x))
\end{equation}
for almost all $x$. Conversely, every lower-semicontinuous energy $F$ as in (\ref{equattro}) with a surface energy density $\varphi$ such that (\ref{eotto}) holds for almost every $x$ is the $\Gamma$-limit of an Ising energy with coefficients $\{c^\e_{i,\xi}\}$ with continuum local volume fraction $\theta$. This localization result turns out much more complex than the one in the elliptic case both because surface energies are not characterized by a single cell problem formula and  because their values must be characterized along $d-1$-hypersurfaces for ${\mathcal H}^{d-1}$-almost all values of $x$.

\medskip

The paper is organized as follows: in Section 2 we fix some notation and introduce the general setting of the problem. In Section 3 we prove Theorem \ref{BoundsTheorem}, which shows the optimality of the bounds in the periodic case. The proof holds with a direct construction when the target energy density is in a dense class of crystalline energy densities, and it is proved by approximation in the general case. It is interesting to note that, in order to recover a system of discrete interactions, it is convenient to interpret homogenized surface energy densities in a $W^{1,1}$ setting, where the extension gives a convex integrand. In Section 4 we prove the localization principle, which is subdivided in Theorems \ref{Localization 2} and \ref{Localization 1}. In their proof we make use of representation and blow-up arguments. In particular, in order to recover (\ref{eotto}) we use the results in \cite{BFM}, which provide a blow-up formula for the limit energy density at all points.

\section{Notation and Setting of the Problem}
\subsection{Preliminaries}
In what follows $\Omega$ will denote a bounded open set of $\mathbb{R}^d$ with Lipschitz boundary. We denote by $\mathcal{A}(\Omega)$ the set of all open subsets contained in $\Omega$. Given $T \subset \mathbb{R}$ we define for fixed $\varepsilon > 0$ the set of functions
\begin{align*}
\mathcal{PC}_\varepsilon(\Omega,T) : = \{ u : \varepsilon \mathbb{Z}^d \cap \Omega \to  T\}.
\end{align*}
We omit the dependence on $T$ when $T=\mathbb{R} $, i.e.~$\mathcal{PC}_\varepsilon(\Omega)=\mathcal{PC}_\varepsilon(\Omega,\mathbb{R})$ as well as $\varepsilon=1$, i.e.~$\mathcal{PC}_1(\Omega,T) = \mathcal{PC}(\Omega,T)$.
In order to carry on our analysis it is convenient to regard $\mathcal{PC}_\varepsilon(\Omega,\{\pm 1\}) $ as a subset of $L^1(\Omega)$. To this end we will identify a function  $u \in \mathcal{PC}_\varepsilon(\Omega,\{\pm 1\}) $ with its piecewise-constant interpolation on the $\varepsilon$-cubes centered in the lattice, still denoted by $u$. More precisely, we set $u(z) = 0$ if $z \in \varepsilon \mathbb{Z}^d \setminus \Omega$ and $u(x) = u(z^\varepsilon_x)$, where $z_x^\varepsilon$ is the closest point in $\varepsilon \mathbb{Z}^d$ to $x$ (which is uniquely defined up to a set of zero measure). Other similar interpolations could be taken into account, actually not affecting our asymptotic analysis. Moreover, setting $H^\pm_\nu(x) = \{y \in \mathbb{R}^d : \pm \langle y-x ,\nu\rangle >0\}$ and omitting the dependence on $x$ if $x=0$, we define
\begin{align*}
u_{x,\nu}(z) =\begin{cases}
1 & x \in H^+_\nu(x)\\
-1& x \in H^-_\nu(x).
\end{cases}
\end{align*}
We set $\Pi_\nu(x)=\{ y \in \mathbb{R}^d : \langle y-x,\nu\rangle =0\}$. and we set $\Pi_\nu^\xi(x)=\{y \in \mathbb{R}^d : 0 \leq \langle y-x,\nu\rangle < \langle \xi,\nu\rangle\}$.
For $R > 0 $ we denote by $B_R(x) =\{ y \in \mathbb{R}^d : |y-x| < R\}$ the open ball with radius $R$ centered in $x$ and we omit the dependence on $x$, when $x=0$, i.e.~$B_R(0)=B_R$. Furthermore we set $B_{R,\nu}^{\pm} = H^\pm_\nu \cap B_R $ We denote with $w_{d-1}$ the $d-1$-dimensional measure of the $d-1$-dimensional unit ball. 
$Q$ denotes the $d$-dimensional unit open cube centered at $0$, $Q=\{ x\in \mathbb{R}^d : |\langle x, e_i \rangle |< \frac{1}{2}, \text{ for all } i \in \{1,\cdots,d\}\}$, whereas we denote by $Q(x_0)$ the cube centered at $x_0$, i.e.~$Q(x_0)=x_0 +Q$. Let $R_\nu\in SO(d)$ be a rotation such that $R(e_n)=\nu$. We denote by $Q^\nu= \{R_\nu(x) : x \in Q\}$ the unit cube with sides either parallel or orthogonal to $\nu$, whereas $Q^\nu(x)= x_0 + Q^\nu$ the unit cube centered at $x_0$ with sides either parallel or orthogonal to $\nu$ and $Q^\nu_\rho(x)= \rho Q^\nu + x_0$ the cube centered at $x_0$ with side lengths $\rho$ and sides either parallel or orthogonal to $\nu$. Given $A$ open bounded with lipschitz boundary, $u \in BV(A)$ we set $\mathrm{tr}(u) \in L^1(\partial(A)) $ the inner trace of the function $u$ on the boundary of $A$.   We say that $\nu \in S^{d-1}$ is rational if there exists $\lambda \in \mathbb{R}$ such that $\lambda \nu \in \mathbb{Z}^d$.
\subsection{Setting of the Problem}
 Let $V\subset \mathbb{Z}^d$ be a finite set containing the standard orthonormal basis $\{e_j\}^d_{j =1}$. We consider a discrete system of long-range interactions with coefficients $c_{i,\xi}\geq 0, i \in \mathbb{Z}^d,\xi \in V$,  The corresponding ferromagnetic spin energy is
\begin{align}\label{1}
E(u) = \frac{1}{4} \sum_{i \in \mathbb{Z}^d}\sum_{\xi \in V} c_{i,\xi}(u_i-u_{i+\xi})^2,
\end{align}
where $u : \mathbb{Z}^d \to \{\pm 1\}, u_i = u(i)$, where $\frac{1}{4}$ is a normalization factor. Such energies correspond to inhomogeneous surface energies in the continuum.
\begin{definition}\label{macrophi}
Let $\{c_{i,\xi}\}, i \in \mathbb{Z}^d,\xi \in V $ be coefficients as above with 
\begin{align*}
\inf_{i \in \mathbb{Z}^d, j \in \{1,\cdots,d\}} c_{i,e_j} >0.
\end{align*}
Then we define the \em{macroscopic energy density} of $\{c_{i,\xi}\} $  as $\varphi : \mathbb{R}^d \times \mathbb{R}^d \to [0,+\infty)$ such that for all $x \in \mathbb{R}^d$, $\varphi(x,\cdot) $ is positively one homogeneous of degree one and for all $\nu \in S^{d-1}, x \in \mathbb{R}^d$ we have
\begin{align}
\begin{split}
\varphi(x,\nu)= \limsup_{R \to +\infty} \frac{1}{4w_{d-1}R^{d-1}} \inf \bigg\{  \sum_{i \in \mathbb{Z}^d \cap B_R(x)}\sum_{\xi \in V} c_{i,\xi}(u_i-&u_{i+\xi})^2 : u \in \mathcal{PC}(\mathbb{R}^d,\{\pm 1\}),\\ &u(i) = u_{x,\nu}(i), i \notin B_R(x)\bigg\}
\end{split}
\end{align}

\end{definition}
\begin{remark}\label{gamma} \rm
The definition above can be interpreted in terms of a passage from a discrete to a continuum description as follows. We consider the scaled energies on $\Omega $
\begin{align*}
E_\varepsilon(u) = \frac{1}{4} \sum_{i,i+\xi \in \Omega_\varepsilon}\sum_{\xi \in V} \varepsilon^{d-1}c_{i,\xi}^\varepsilon(u_{\varepsilon i}-u_{\varepsilon(i+\xi)})^2
\end{align*}
where $u : \varepsilon\mathbb{Z}^d \cap \Omega \to \{\pm1 \}$ and $\Omega_\varepsilon=\mathbb{Z}^d \cap (\frac{1}{\varepsilon}\Omega)$. Identifying $u$ with its piecewise constant interpolation $u \in \mathcal{PC}_\varepsilon(\Omega)$, we can regard this energies defined on $L^1(\Omega)$. Their $\Gamma$-limit in that space is finite only on $BV(\Omega,\{\pm 1\})$, where it has the form
\begin{align*}
E_\varphi(u) = \int_{\partial^*\{u=1\} \cap \Omega} \varphi(x,\nu_u(x))\mathrm{d}\mathcal{H}^{d-1}
\end{align*}
with $\varphi$ as above  (\cite{AS},\cite{BP}).
\end{remark}
\section{The Periodic Case}
In this section we will consider the case where $\{c_{i,\xi}\}$ is periodic, i.e.~there exists $T \in \mathbb{N}$ such that for all $i \in \mathbb{Z}^d, \xi \in V$ we have
\begin{align*}
c_{(i+Te_j),\xi} = c_{i,\xi} \text{ for all } j \in \{1,\cdots,d\}
\end{align*}
 and for $\alpha = (\alpha_\xi)_{\xi \in V}, \beta = (\beta_\xi)_{\xi \in V}$
\begin{align}
c_{i,\xi} \in \{\alpha_\xi,\beta_\xi\} \text{ with } 0 < \alpha_\xi < \beta_\xi;
\end{align}
\begin{remark}\rm Let $\{c_{i,\xi}\}, i \in \mathbb{Z}^d,\xi \in V$ be periodic coefficients. Then the {\em macroscopic energy density} of $\{c_{i,\xi}\}, i \in \mathbb{Z}^d,\xi \in V$ reduces to a {\em homogenized energy density} of $\{c_{i,\xi}\}, i \in \mathbb{Z}^d,\xi \in V$ namely a convex positively homogeneous function of degree one $\varphi:\R^d\to[0,+\infty)$ such that for all $\nu\in S^{d-1}$ we have
\begin{equation}\label{phi}
\begin{split}
\varphi(\nu)= \lim_{R \to +\infty} \frac{1}{4w_{d-1}R^{d-1}} \inf \Big\{  \sum_{i \in \mathbb{Z}^d \cap B_R}\sum_{\xi \in V} c_{i,\xi}(u_i-u_{i+\xi})^2 : &u \in \mathcal{PC}(\mathbb{R}^d,\{\pm 1\}),\\ &u(i) = u_{0,\nu}(i), i \notin B_R\Big\}
\end{split}
\end{equation}
This is true thanks to \cite{AS}.
\end{remark} 
If we have such coefficients, we define the {\em volume fraction} of $\beta_\xi$-bonds and the total volume fraction, respectively, as
\begin{equation} \label{theta}
\begin{split}
&\theta_\xi(\{c_{i,\xi}\}) = \frac{1}{T^d}\#\{ i \in \mathbb{Z}^d : i \in [0,T)^d,c_{i,\xi}=\beta_\xi\}, \\
&\theta(\{c_{i,\xi}\})= \frac{1}{\#V} \sum_{\xi \in V} \theta_\xi(\{c_{i,\xi}\}).
\end{split}
\end{equation}
\begin{definition}
Let $\theta \in [0,1]$. The set of \em{homogenized energy densities of mixtures of $\alpha$ and $\beta$ bonds corresponding to $V$, with volume fraction $\theta$ (of $\beta$ bonds) } is defined as
\begin{equation}
\begin{split}
\textbf{H}_{\alpha,\beta,V}(\theta)=\{ &\varphi :\mathbb{R}^d \to [0,+\infty) : \text{ there exist } \theta^k \to \theta, \varphi^k \to \varphi \text{ and } \{c^k_{i,\xi}\}, \\ &\text{with } \theta(\{c^k_{i,\xi}\}) = \theta^k \text{ and } \varphi^k \text{ homogenized energy density of } \{c^k_{i,\xi}\}\}.
\end{split}
\end{equation}
\end{definition}
The following theorem completely characterizes the set $\textbf{H}_{\alpha,\beta,V}(\theta)$.
\begin{theorem}[Optimal bounds]\label{BoundsTheorem} The elements of the set $\textbf{H}_{\alpha,\beta,V}(\theta)$ are all the even, convex positively homogeneous functions of degree one $\varphi : \mathbb{R}^d \to [0,+\infty)$ such that
\begin{equation}
\sum_{\xi \in V} \alpha_\xi |\langle \nu,\xi\rangle| \leq \varphi(\nu) \leq \sum_{\xi \in V} (\theta_\xi \beta_\xi + (1-\theta_\xi)\alpha_\xi)|\langle \nu,\xi \rangle|
\end{equation}
for some $\theta_\xi \in [0,1] ,\xi \in V$ such that
\begin{equation} \label{Averagefraction}
\frac{1}{\# V}\sum_{\xi \in V} \theta_\xi = \theta.
\end{equation}
\end{theorem}
Note that the lower bound for functions in $\textbf{H}_{\alpha,\beta,V}(\theta)$ is independent of $\beta$. This follows by a comparison argument from \cite{ABC} in the case for nearest-neighbour spin systems; the more refined argument  in Proposition \ref{BbP} will give the optimality of the lower bound in the general case. 

In the case $\theta =1$ we have all functions satisfying the trivial bounds
\begin{equation} \label{tri}
\sum_{\xi \in V} \alpha_\xi |\langle \nu,\xi\rangle| \leq \varphi(\nu) \leq \sum_{\xi \in V} \beta_\xi |\langle \nu,\xi \rangle|
\end{equation}
This is due to the fact that in that case by considering $\theta^k \to 1$ we allow a vanishing volume fraction of $\alpha$ bonds, which is nevertheless sufficient to allow for all such $\varphi$.

\begin{figure}[h!]
\centerline{\includegraphics [width=3.67in]{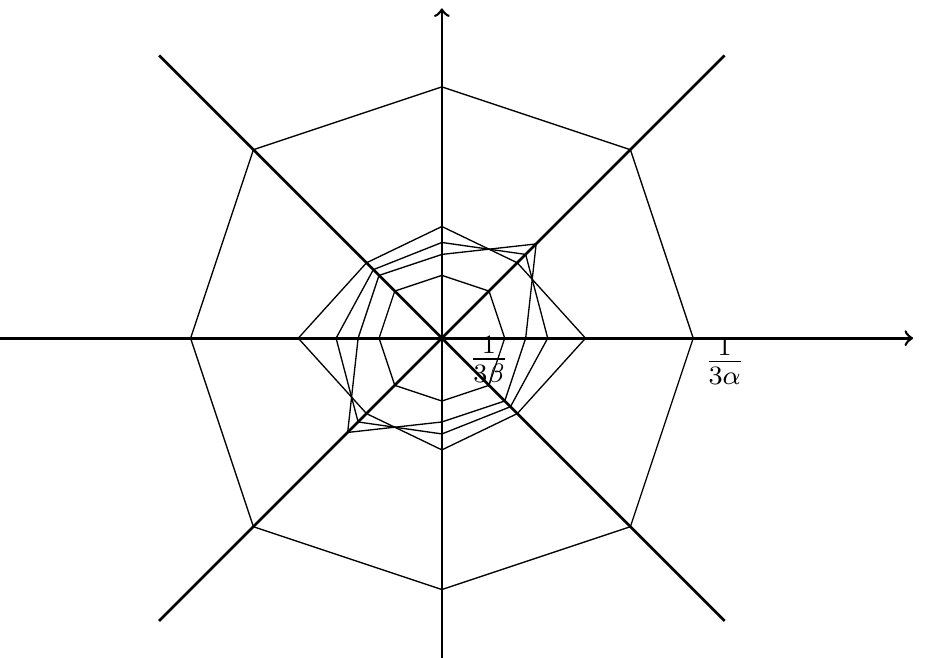}\includegraphics [width=3.5in]{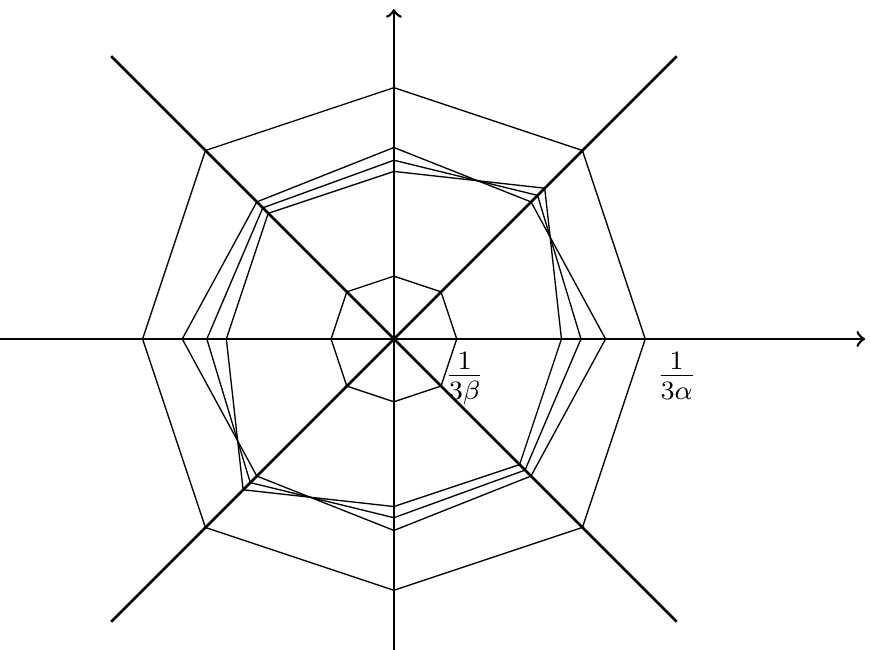}}
\caption{Level sets of $\varphi$ in the case $\theta=\frac{1}{2}$ and $\theta < \frac{1}{2}$, respectively}\label{f1}
   \end{figure}

\begin{example}\rm
We consider the two-dimensional case with nearest and next-to-nearest neighbour interactions; i.e., choosing $V= \{e_1,e_2,e_1+e_2=:v_1,e_1-e_2=:v_2\}$. Theorem \ref{BoundsTheorem} states that the homogenized densities $\varphi$ in $\textbf{H}_{\alpha,\beta,V}(\theta)$ 
have to satisfy 
\begin{align*}
\alpha (|\nu_1| +|\nu_2| + |\langle \nu, v_1\rangle| +|\langle \nu, v_2\rangle| \leq \varphi(\nu) \leq  c_{1,1}|\nu_1| + c_{1,2}|\nu_2| + c_{2,1}|\langle \nu, v_1\rangle| +c_{2,2}|\langle \nu, v_2\rangle|,
\end{align*}
where
\begin{align*}
c_{1,1}+c_{1,2}= \theta_1 \beta + (1-\theta_1)\alpha,\qquad c_{2,1}+c_{2,2}= \theta_2 \beta + (1-\theta_2)\alpha,
\end{align*}
and $\theta_1 +\theta_2=2\theta$.  This is pictured in terms of the sublevelsets $\{\varphi \leq 1\}$ on the left-hand side of Fig.~\ref{f1} for $\theta\ge\frac{1}{2}$ and on the right-hand side of Fig.~\ref{f1} for $\theta < \frac{1}{2}$. Such sublevelsets are convex and even.
Note that by the trivial bounds such sublevelsets have always to be contained in the sublevelset of $\varphi$ with all the coefficients equal to $\alpha$  and contain always the sublevelset of $\varphi$ with all the coefficients equal to $\beta$. Those are the large and the small regular octagons in the figures. Further, since the $\varphi$ have to satisfy the upper bounds, they have to contain at least one of the even octagons whose vertices lie on the straight lines with coordinates 
\begin{align*}
&x_1=\pm \frac{1}{2(\theta_1+\theta_{2,1})(\beta-\alpha)+4\alpha},&&x_2=x_1 \\
&x_1=\pm\frac{1}{(2\theta_2+\theta_{1,1})(\beta-\alpha)+3\alpha}, &&x_2=0\\
&x_1=0,&&x_2=\pm\frac{1}{(2\theta_2+\theta_{1,2})(\beta-\alpha)+3\alpha} \\
&x_1=\pm\frac{1}{2(\theta_1+\theta_{2,2})(\beta-\alpha)+4\alpha}, &&x_2=-x_1,\\
\end{align*}
where $\theta_{1,1}, \theta_{1,2}$ are the vertical and horizontal volume fractions of the nearest-neighbor bonds satisfying $ \theta_{1,1}+\theta_{1,2}=2\theta_1$ and $\theta_{2,1},\theta_{2,2}$ are the volume fractions corresponding to $v_2$ and $v_1$ respectively and satisfying  $ \theta_{2,1}+\theta_{2,2}=2\theta_2$.
Those are represented by the thick lines.
\end{example}

\subsection{Derivation of Bounds}
We now derive the bounds of Theorem \ref{BoundsTheorem} by using the following proposition.
\begin{proposition}[Bounds by averaging]\label{BbA} Let $\varphi$ be the \rm{ homogenized energy density of $\{c_{i,\xi}\}$} as  in (\ref{phi}); then we have
\begin{equation}
\varphi(\nu) \leq \sum_{\xi \in V} (\theta_\xi \beta_\xi + (1-\theta_\xi)\alpha_\xi) |\langle \nu, \xi\rangle|.
\end{equation}
\end{proposition}
\begin{proof}
The proof is obtained by constructing a suitable competitor $u \in \mathcal{PC}(\mathbb{R}^d,\{\pm 1\})$, $u(i) = u_\nu(i), i \notin B_R$ in the definition of (\ref{phi}). To that end define for $j \in [0,T)^d \cap \mathbb{Z}^d$  $u_j \in \mathcal{PC}(\mathbb{R}^d,\{\pm 1\})$, $u(i) = u_\nu(i), i \notin B_R$ by
\begin{align*}
u_j(i) =\begin{cases}
u_\nu(i) &\text{if } i \notin B_R\\
u_{j,\nu}(i) &\text{if } i \in B_R.
\end{cases}
\end{align*}
Let $j_0 \in [0,T)^d \cap \mathbb{Z}^d$ be such that
\begin{align*}
\frac{1}{4}\sum_{\xi \in V} \sum_{i \in \mathbb{Z}^d \cap B_R} c_{i,\xi}((u_{j_0})_i-(u_{j_0})_{i+\xi})^2\leq \frac{1}{T^d}\frac{1}{4}\sum_{j \in [0,T)^d\cap \mathbb{Z}^d}\sum_{\xi \in V} \sum_{i \in \mathbb{Z}^d \cap B_R} c_{i,\xi}((u_{j})_i-(u_{j})_{i+\xi})^2.
\end{align*}
we have
\begin{align*}
&\frac{1}{T^d}\frac{1}{4}\sum_{j \in [0,T)^d\cap \mathbb{Z}^d}\sum_{\xi \in V} \sum_{i \in \mathbb{Z}^d \cap B_R} c_{i,\xi}((u_{j})_i-(u_{j})_{i+\xi})^2\\ =&\frac{1}{T^d}\sum_{j \in [0,T)^d\cap \mathbb{Z}^d}\sum_{\xi \in V} \sum_{i \in \Pi_\nu^\xi(j) \cap B_R \cap \mathbb{Z}^d} c_{i,\xi}\\ = &\frac{1}{T^d}\sum_{\xi \in V} \sum_{i \in [0,T)^d \cap \mathbb{Z}^d}\sum_{j \in [0,T)^d \cap \mathbb{Z}^d}c_{i,\xi} \# \{ i' \in (T\mathbb{Z}^d + i) \cap \Pi_\nu^\xi(j) \cap B_R\} \\\leq &\frac{1}{T^d}\sum_{\xi \in V} \sum_{i \in [0,T)^d \cap \mathbb{Z}^d} c_{i,\xi} w_{d-1}R^{d-1} |\langle \nu,\xi\rangle| + o(R^{d-1}) \\=&\sum_{\xi \in V} (\theta_\xi \beta_\xi + (1-\theta_\xi)\alpha_\xi)|\langle \nu,\xi\rangle| + o(R^{d-1}).
\end{align*}
The last equality follows from splitting the sum into the two sets where $c_{i,\xi} = \alpha_\xi $ or $\beta_\xi$ respectively.
Since $u_{j_0}$ is admissible in the definition of (\ref{phi}), dividing by $R^{d-1}$ and letting $R \to \infty$ yields the claim.
\end{proof}
Proposition \ref{BbA} together with the trivial bound from below gives the bounds in the statement of Theorem \ref{BoundsTheorem}. in the following section we prove their optimality.
\subsection{Optimality of Bounds}
We introduce some notation for the this section. Let $\Xi=\{ \xi_1,\cdots,\xi_d\} \subset \mathbb{Z}^d$ be an orthogonal basis and let $z \in \mathbb{Z}^d$ and let
\begin{align*}
\mathcal{L}_z(\Xi) := \Big\{ i \in \mathbb{R}^d : i = z + \sum_{k=1}^d \lambda_k \xi_k; \lambda_k \in \mathbb{Z}\Big\}, \quad \mathcal{L}_z^T (\Xi)= z + T\mathcal{L}_0(\Xi)
\end{align*}
and for $j=1,\cdots,d$ we set
\begin{align*}
\mathcal{L}_{z,j}(\Xi) := \Big\{ i \in \mathbb{R}^d : i = z + \sum_{k=1}^d \lambda_k \xi_k; \lambda_k \in \mathbb{Z}, \lambda_j =0\Big\}, \quad
\mathcal{L}_{z,j}^T(\Xi) = z + T\mathcal{L}_{0,j}(\Xi)
\end{align*}
the projection of the lattice into the plane orthogonal to $\xi_j$. We set 
\begin{align*}
&P_z(\Xi) :=\Big\{i \in \mathbb{R}^d : i =z + \sum^d_{j=1}\lambda_j \xi_j; \lambda_j \in [0,1)\Big\}
\end{align*}
the fundamental parallelepiped spanned by the vectors $\Xi$ translated to the point $z$, and
\begin{align*}
P_z^T(\Xi) :=\Big\{i \in \mathbb{R}^d : i =z + T\sum^d_{j=1}\lambda_j \xi_j; \lambda_j \in [0,1)\Big\}.
\end{align*}

 Set for $\{c_{i,\xi_j}\}_{i \in \mathcal{L}_z(\Xi), j \in \{1,\cdots,d\}}$ with $c_{i,\xi_j} \geq 0$ and $T$-periodic, i.e.~$c_{i+T\xi_j,\xi_k}=c_{i,\xi_k}$ for all $j,k \in \{1,\cdots,d\}$
\begin{equation} \label{Latticephi}
\begin{split}
\varphi_{\Xi,z}(\nu)= \limsup_{R \to +\infty} \frac{1}{4w_{d-1}R^{d-1}} \inf \Big\{  &\sum_{i \in \mathcal{L}_z(\Xi) \cap B_R}\sum^d_{j=1} c_{i,\xi_j}(u_i-u_{i+\xi_j})^2;  \\& u : \mathcal{L}_z(\Xi) \to \{\pm 1\} ,u(i) = u_{0,\nu}(i), i \notin B_R\Big\}
\end{split}
\end{equation} 
(we omit the dependence on $\{c_{i,\xi_j}\}$). By regrouping the interactions $\{c_{i,\xi}\}$  of energy (\ref{1})  on sublattices $\mathcal{L}_{z,j}(\Xi)$ we will use (\ref{Latticephi}) to obtain a lower bound in (\ref{phi}).
(Note that possibly one has to set $c_{i,\xi_j}=0$ for some $j \in \{1,\cdots,d\}$ if $\xi_j \notin V$). 

Note that if $c_{i,\xi}$ is $T$-periodic along the coordinate directions, then for every $\eta \in \mathbb{Z}^d$ there exists $T'=T'_\eta$ such that $c_{i+T'\eta,\xi}=c_{i,\xi}$. 
 \begin{proposition}[Bounds by projection] \label{BbP} Let $\Xi=\{ \xi_1,\cdots,\xi_d\}$ be an orthogonal basis, $z \in \mathbb{Z}^d$ and $\{c_{i,\xi_j}\}_{i \in \mathcal{L}_z(\Xi), j\in \{1,\cdots,d\}} $ be non-negative coefficients. Let $\varphi_{\Xi,z} : \mathbb{R}^d \to [0,+\infty) $ be the even convex positively homogeneous function of degree one given by (\ref{Latticephi}), then 
\begin{equation}
\varphi_{\Xi,z}(\nu) \geq \sum^d_{j=1} c^p_j | \langle\nu,\xi_j \rangle|
\end{equation}
where 
\begin{equation}
c^p_j= \frac{1}{T^{d-1}|P_z(\Xi)|}\sum_{k \in \mathcal{L}_{z,j}(\Xi)\cap P_z^T(\Xi)} \min \{ c_{i,\xi_j} : i-k = \lambda \xi_j \text{ for some } \lambda \in \mathbb{Z} \}
\end{equation}
(the letter $p$ in $c^p_j$ stands for {\em projection}).
\end{proposition}

\begin{proof} Let $u : \mathcal{L}_z \to \{ \pm1\} $ be such that $u(i)= u_\nu(i) $ for all $i \notin B_R$. Set for $j =1,\cdots, N$
\begin{align*}
I_j := \Big\{ i \in \mathcal{L}_{z,j}(\Xi) : \{i + t\xi_j : t \in \mathbb{R}\} \cap B_{R,\nu}^{\pm} \neq \emptyset\Big\}.
\end{align*} 
Noting that for all $k \in I_j$ there exists at least $i \in \{k + \lambda \xi : \lambda \in \mathbb{Z}\} \cap B_R$,  such that $u_i \neq u_{i+\xi}$, we have
\begin{align*}
&\frac{1}{4}\sum_{i \in \mathcal{L}_z \cap B_R}\sum^d_{j=1} c_{i,\xi_j}(u_i-u_{i+\xi_j})^2 \geq \sum_{j=1}^d\sum_{k \in I_j } \min \{ c_{i,\xi_j} : i-k = \lambda \xi_j \text{ for some } \lambda \in \mathbb{Z}\} \\&\geq \sum_{j=1}^d\sum_{k \in \mathcal{L}_{z,j}(\Xi) \cap P_z^T(\Xi)}\min \{ c_{i,\xi_j} : i-k = \lambda \xi_j \text{ for some } \lambda \in \mathbb{Z}\} \#  (T\mathcal{L}_{k,j}(\Xi)\cap I_j) \\&\geq \sum_{j=1}^d\sum_{k \in \mathcal{L}_{z,j}(\Xi) \cap P_z^T(\Xi)}\min \{ c_{i,\xi_j} : i-k = \lambda \xi_j \text{ for some } \lambda \in \mathbb{Z}\} \frac{|\langle \nu,\xi_j\rangle|}{|P_{z,j}(\Xi)|}\frac{w_{d-1}R^{d-1}}{T^{d-1}} +o(R^{d-1}).
\end{align*}
Where we have used the fact that
\begin{align*}
\# ( T\mathcal{L}_{k,j} \cap I_j) &= \frac{\mathcal{H}^{d-1}\Bigl((\Pi_\nu \cap B_R) \text{ projected onto } z + \Pi_{\frac{\xi_j}{||\xi_i||}}\Bigr)}{\mathcal{H}^{d-1}\Bigl(P_z^T(\Xi) \cap \Big(z+ \Pi_{\frac{\xi_j}{||xi_j||}}\Big)\Bigr)}\\&= \frac{\frac{1}{||\xi_j||}|\langle \nu,\xi_j\rangle|}{\underset{i \neq j}{\Pi^{d}_{i=1}||\xi_i||}}\frac{w_{d-1}R^{d-1}}{T^{d-1}} +o(R^{d-1}).
\end{align*}
Taking the infimum over $u : \mathcal{L}_z \to \{\pm1\}$ such that $u(i)=u_\nu(i), i \notin B_R$, dividing by $w_{d-1}R^{d-1}$ and letting $R \to \infty$ yields the claim.
\end{proof}
We will now use Proposition \ref{BbP} to prove the optimality of bounds.
First we deal with a special case, from which the general result will be deduced by
approximation.
\begin{proposition}\label{Special} Let $V$ be as in Theorem \ref{BoundsTheorem} and 
\begin{align*}
\psi(\nu) = \sum_{\xi \in V} c_\xi |\langle \nu, \xi\rangle|
\end{align*}
with $\alpha_\xi \leq c_\xi \leq \beta_\xi, \xi \in V$ such that
\begin{align*}
c_\xi=t_\xi \beta_\xi + (1-t_\xi)\alpha_\xi \leq \theta_\xi \beta_\xi + (1-\theta_\xi)\alpha_\xi
\end{align*}
where $\theta_\xi,t_\xi \in (0,1) \cap \mathbb{Q}, \xi \in V$.
 Then there exist $T\in \mathbb{N}$ and $\{c_{i,\xi}\}$ $T$-periodic with $\theta_\xi(\{c_{i,\xi}\})=\theta_\xi$ and $\psi$ is  the {\em homogenized energy density} of $\{c_{i,\xi}\}$. In particular if $\theta$ satisfies (\ref{Averagefraction}) with $\theta_\xi$ then $\psi \in \textbf{H}_{\alpha,\beta,V}(\theta)$.
\end{proposition}
\begin{proof}
We construct $\{c_{i,\xi}\}$ with some period $T \in \mathbb{N}$ and 
\begin{align*}
\theta_\xi(\{c_{i,\xi}\})= \theta_\xi \text{ for all } \xi \in V
\end{align*}
by defining the bonds separately for each direction of interaction $\xi \in V$. Note that if we construct $T_\xi$-periodic coefficients for each $\xi \in V$ there exists a common $T \in \mathbb{N}$ such that the coefficients $\{c_{i,\xi}\}$ are $T$-periodic. 

For $\xi \in V$, let $\Xi=\{ \xi_1,\cdots,\xi_d=\xi\} \subset \mathbb{Z}^d$ be an orthogonal basis and $z \in \mathbb{Z}^d \cap P_0(\Xi)$.   Set for $T \in \mathbb{N}$, $\nu = \frac{v}{||v||}$ for some $v \in V$ and $\Xi$ an orthogonal basis
\begin{align*}
A_{\nu}^z(T,\Xi):=\Big\{ i \in \mathcal{L}_{z}(\Xi) : \{i+t\xi : t\in [0,1)\} \cap \bigcup_{j \in \mathcal{L}^T_0(\Xi)} (j + \Pi_\nu) \neq \emptyset \Big\}.
\end{align*}
This is the minimal $T$-periodic set of points in the lattice $\mathcal{L}_z(\Xi)$ interacting in direction $\xi$ when we use $u_{\nu}$ as a test function.
Note that for $i \in \mathcal{L}_{z,d}(\Xi) $ we have  
\begin{align*}
\#\Big(A^z_{\nu}(T,\Xi) \cap \{i +t\xi : t\in \mathbb{R}\} \cap P_z^T(\Xi) \Big) \leq C(\nu,\xi)
\end{align*} 
for all $\langle\nu,\xi \rangle \neq 0$, $\nu \in S^{d-1}$ rational, $\xi \in \mathbb{Z}^d $.
In fact, if $ \nu $ is rational and $\Xi$ is the standard orthonormal basis we have that for $\{\nu,\nu_1,\ldots,\nu_{d-1}\}$ with $\nu_i \in \mathbb{Z}^d$ (which can be chosen since $\nu$ is rational) one can choose  $C(\nu,\xi)\leq \Pi_{i=1}^{d-1} ||\nu_i||_1$ and the general case can be reduced to this one by a change of coordinate which preserves the rationality of $\nu$.
Choose $T \in \mathbb{N}$ such that $T^{d}\theta_\xi \in \mathbb{N} $, $T^{d-1}(1-t_\xi)=N_\xi \in \mathbb{N}$ and
\begin{align} \label{card}
(1-t_\xi)\underset{\langle v, \xi\rangle \neq 0}{\sum_{v\in V}} C\Big(\frac{v}{||v||},\xi\Big) \leq  T(1-\theta_\xi)
\end{align}  
 for all $\xi \in V$.
Choose  $ A_\xi \subset \mathcal{L}_{z,d}(\Xi) \cap P_z^T(\Xi)  $ such that $\# A_\xi = N_\xi$.
 We define
\begin{align*}
c_{i,\xi} = \begin{cases} \alpha_\xi &i = \lambda\xi +i' ,\lambda \in \mathbb{Z}, i' \in A_\xi, i \in A^z_{\nu}(T,\Xi) \\
\beta_\xi &i = \lambda\xi +i' ,\lambda \in \mathbb{Z}, i' \in(\mathcal{L}_{z,d}(\Xi) \cap P_z^T(\Xi))\setminus A_\xi 
\end{cases}
\end{align*}
and any choice of $\alpha_\xi$ and $\beta_\xi$ for other indices $i$ only subjected to the total constraint that $\theta_\xi(\{c_{i,\xi}\})=\theta_\xi$, which is possible, since due to (\ref{card}) it holds 
\begin{align*}
\#\Big\{ i = \lambda\xi +i' ,\lambda \in \mathbb{Z}, i' \in A_\xi, i \in A^z_{\nu}(T,\Xi) \Big\}&\leq N_\xi \underset{\langle v, \xi \rangle \neq 0}{\sum_{v \in V}} C\Big(\frac{v}{||v||},\xi\Big) \\&= T^{d-1}(1-t_\xi) \underset{\langle v, \xi \rangle \neq 0}{\sum_{v \in V}}C\Big(\frac{v}{||v||},\xi\Big) \leq (1-\theta_\xi)T^d
\end{align*}
and
\begin{align*}
\#\Big\{i = \lambda\xi +i' ,\lambda \in \mathbb{Z}, i' \in(\mathcal{L}_{z,d}(\Xi) \cap P_z^T(\Xi))\setminus A_\xi \Big\}= (T^{d-1}-N_\xi)T= t_\xi T^d \leq \theta_\xi T^d.
\end{align*}
With this choice of $c_{i,\xi}$ we have 
\begin{equation*}
\min\{c_{i,\xi} : i = \lambda \xi +k \text{ for some } \lambda \in \mathbb{Z}\} =\begin{cases} \alpha_\xi
&\text{if } k \in A_\xi\\
\beta_\xi &\text{if } k \in (\mathcal{L}_{z,d}(\Xi) \cap P_z^T(\Xi))\setminus A_\xi.
\end{cases}
\end{equation*}
Hence, Proposition \ref{BbP} yields that the homogenized energy density $\varphi$ of $\{c_{i,\xi}\}$ satisfies 
\begin{align*}
\varphi(\nu) &\geq \sum_{\xi \in V} \sum_{z \in P_z(\Xi)}\frac{1}{|P_0(\Xi)|} (t_\xi\beta_\xi +(1-t_\xi)\alpha_\xi)|\langle\nu,\xi\rangle| \\&\geq  \sum_{\xi \in V} (t_\xi\beta_\xi +(1-t_\xi)\alpha_\xi)|\langle\nu,\xi\rangle|=\psi(\nu)
\end{align*} 
as desired.
To give an upper bound let $v\in V$ and set $\nu = \frac{v}{||v||}\in S^{d-1}$. Testing (\ref{phi}) with $u_\nu$ we have that
\begin{align*}
\varphi(\nu) &\leq \lim_{R \to \infty} \frac{1}{w_{d-1}R^{d-1}}\sum_{\xi \in V}\sum_{i \in \mathbb{Z}^d \cap B_R} c_{i,\xi}((u_\nu)_i - (u_\nu)_{i+\xi})^2\\& \leq \sum_{\xi \in V} (t_\xi\beta_\xi +(1-t_\xi)\alpha_\xi)|\langle \nu,\xi\rangle|=\psi(\nu)
\end{align*}
and, since $\varphi$ is a convex positively homogeneous function of degree one and $\psi$ is the greatest convex positively homogeneous function $g$ of degree one  such that $g(\frac{v}{||v||})\leq \psi(\frac{v}{||v||})  $ for all $v \in V$, the desired equality.
\end{proof}
The next proposition shows, that for any $\varphi$ satisfying the bounds of Theorem \ref{BoundsTheorem} its associated surface energy $E_\varphi$ as in Remark \ref{gamma} is  the $\Gamma$-limit of energies of the type (\ref{1}) where we have that the period $T^\varepsilon$ of the interaction coefficients $ \{c^\varepsilon_{i,\xi}\} $ of the approximating energies goes to $+\infty$. In the proof of Theorem \ref{General} we use the existence of such an approximating sequence of energies to deduce the convergence of their homogenized energy densities $\varphi^\varepsilon$ to the limit energy density $\varphi$.
\begin{proposition}\label{Gammaclosure}
Let $\varphi : \mathbb{R}^d \to [0,+\infty)$ be convex, even, positively 1-homogeneous and such that
\begin{align}
\sum_{\xi\in V}\alpha_\xi|\langle \nu,\xi\rangle| \leq \varphi(\nu) \leq \sum_{\xi\in V}(\theta_\xi\beta_\xi+(1-\theta_\xi)\alpha_\xi)|\langle \nu,\xi\rangle|
\end{align}
with $\theta_\xi \in [0,1], \xi \in V$. Then there exist $\{c^\varepsilon_{i,\xi}\}$, $\frac{1}{\varepsilon}$-periodic such that $ \theta_\xi(\{c^\varepsilon_{i,\xi}\}) \to \theta_\xi $ for all $\xi \in V$ and $E_\varepsilon$  $\Gamma$-converges with respect to the strong $L^1(\Omega)$-topology to the functional $E_\varphi$ .
\end{proposition}
\begin{proof}
\underline{Step 1:} We may suppose that
\begin{align}\label{ineq0}
\sum_{\xi \in V} \alpha_\xi |\langle \nu, \xi\rangle| < \varphi(\nu) < \sum_{\xi \in V} c_\xi |\langle \nu, \xi\rangle| = \sum_{\xi \in V} ( \theta_\xi\beta_\xi + (1-\theta_\xi)\alpha_\xi)|\langle \nu, \xi\rangle| 
\end{align}
for some $\theta_\xi \in [0,1], \xi \in V $ satisfying (\ref{Averagefraction}) for some $\theta \in [0,1]$.
Indeed if we have equality in (\ref{ineq0}) we can find $\varphi_k$ satisfying (\ref{ineq0}) strictly and $\varphi_k \to \varphi$ monotonically. Hence by the monotone convergence theorem we have that $E_{\varphi_k}(u) \to E_\varphi(u)$ monotonically for all $u \in BV(\Omega,\{\pm 1\})$, therefore by [\cite{DM}, Proposition 5.4,5.7] we have that $E_{\varphi_k} \, \Gamma$-converge to $E_\varphi$.

 Moreover, we can assume that $\varphi$ is crystalline and the vertices of the set $\{\varphi \leq 1\}$ correspond to rational directions and contain the directions $V$, i.e.~there exists $N \in \mathbb{N}, N \geq \# V$ such that 
\begin{align}\label{crystalline}
\varphi(\nu) = \sum^N_{j=1}c_j |\langle \nu,\nu_j\rangle|,
\end{align}
 and for all $\xi \in V$ there exists $k \in \{1,\cdots, \# V\}$ such that $\lambda_k\nu_k =\xi$.
with $c_j \geq 0$.
Note that this is possible due to an approximation argument that still maintains the bound (\ref{ineq0}).

\medskip

\underline{Step 2:} For every $\varphi$ satisfying (\ref{crystalline}) and (\ref{ineq0}) the functionals $E_{\varepsilon} : BV(\Omega;\{\pm1\}) \to [0,+\infty)$ given by
\begin{align}\label{EepsN}
E_{\varepsilon}(u)= \int_{\partial^*\{u=1\}}f\Big(\frac{x}{\varepsilon},\nu_u(x)\Big)\mathrm{d}\mathcal{H}^{d-1}
\end{align}
where
\begin{align*}
f(y,\nu) = \begin{cases}\varphi(\nu_k) &\text{if } y \in \Pi_{\nu_k} +\mathbb{Z}^d, k =1,\cdots,N\\
\sum_{\xi \in V} c_\xi |\langle\nu,\xi\rangle| &\text{otherwise}
\end{cases}
\end{align*}
$\Gamma$-converge to $E_\varphi$.

In fact by \cite{AI} we have $E_{\varepsilon}$ $\Gamma$-converges as $\varepsilon \to 0$ to $E : BV(\Omega;\{\pm 1\}) \to [0,\infty)$ defined by
\begin{align*}
E(u) = \int_{\partial^* \{u=1\}} f_{\mathrm{hom}}(\nu_u(x))\mathrm{d}\mathcal{H}^{d-1}
\end{align*}
where
\begin{align*}
f_{\mathrm{hom}}(\nu) = \lim_{T \to \infty} \frac{1}{T^{d-1}}\inf \big\{ \int_{\partial^*\{u=1\} \cap TQ_\nu}f(x,\nu_u(x))\mathrm{d}\mathcal{H}^{d-1}: u \in &BV(TQ_\nu;\{\pm1\}), \\&u = u_\nu \text{ on } \partial TQ_\nu\big\}.
\end{align*}
We now prove that $f_{\mathrm{hom}}(\nu) = \varphi(\nu)$. 

We first show that $f_{\mathrm{hom}}(\nu)\leq \varphi(\nu)$. Let $u \in BV(TQ,\{\pm 1\})$ be such that $u=u_\nu$ on $\partial TQ_\nu$ and define $A_j = \Pi_{\nu_j} + \mathbb{Z}^d$. We know that $\nu_j = \pm \nu_u \, \mathcal{H}^{d-1}$-a.e. on $\partial^* \{u=1\} \cap A_j$ and $\varphi(\nu) = \varphi(-\nu)$. We get
\begin{align*}
&\int_{\partial^* \{u=1\} \cap TQ_\nu} f(x,\nu_u(x))\mathrm{d}\mathcal{H}^{d-1}(x)= \\ =&\sum^N_{j =1} \int_{\partial^* \{u=1\} \cap A_j \cap TQ_\nu} f(x,\nu_u(x))\mathrm{d}\mathcal{H}^{d-1}(x) + \int_{(\partial^* \{u=1\} \setminus A_j) \cap TQ_\nu} f(x,\nu_u(x))\mathrm{d}\mathcal{H}^{d-1}(x)  \\
\geq &\sum_{j=1}^N \int_{\partial^* \{u=1\} \cap A_j \cap TQ_\nu} \varphi(\nu_u(x))\mathrm{d}\mathcal{H}^{d-1}(x) + \int_{(\partial^* \{u=1\} \setminus A_j) \cap TQ_\nu} \varphi(\nu_u(x))\mathrm{d}\mathcal{H}^{d-1}(x) \\
=&\int_{\partial^* \{u=1\} \cap TQ_\nu} \varphi(\nu_u(x))\mathrm{d}\mathcal{H}^{d-1}(x) \geq T^{d-1} \varphi(\nu),
\end{align*} 
where the last inequality follows from BV-ellipticity (see \cite{AB}) and a rescaling argument. Thus by the definition of $f^N_{hom}$ we get
$
f_{\mathrm{hom}}(\nu) \geq \varphi(\nu)$. 

Now we deal with the inequality $f_{hom} \leq \varphi$. We have for every $j =1,\cdots,N$
\begin{align*}
\int_{\Pi_{\nu_j} \cap TQ_{\nu^j}} f(x,\nu_u(x))\mathrm{d}\mathcal{H}^{d-1}(x) = T^{d-1} \varphi(\nu_j),
\end{align*}
Now since $\varphi$ is the greatest convex,even positively $1$-homogeneous function $g$ such that $g(\nu_j)\leq \varphi(\nu_j)$ for all $j =1,\cdots,N$ we have that $f_{\mathrm{hom}}(\nu) \leq \varphi(\nu)$ for all $\nu \in S^{d-1}$.

\medskip

\underline{Step 3:} Note that for every $k \in \mathbb{N}$ we can write
\begin{align}
\varphi(\nu_k) = \sum_{\xi \in V} (t_\xi\beta_\xi +(1-t_\xi)\alpha_\xi)|\langle \nu_k,\xi\rangle|= \sum_{\xi \in V} c_\xi^{k}|\langle \nu_k, \xi\rangle|
\end{align}
with $\alpha_\xi<c_\xi^{k}<c_\xi$. We can therefore consider equivalently
\begin{align*}
f(y,\nu)=\begin{cases}\sum_{\xi \in V}c^{k}_\xi|\langle \nu,\xi\rangle| &\text{if } y \in \Pi_{\nu_k} +\mathbb{Z}^d, k=1,\cdots,N\\
\sum_{\xi \in V} c_\xi |\langle \nu,\xi\rangle| &\text{otherwise}.
\end{cases}
\end{align*}
Every functional of the form (\ref{EepsN}) can be approximated by a functional $E_{\delta,\varepsilon} : BV(\Omega,\{\pm 1\}) \to [0,+\infty)$ of the form
\begin{align}\label{EepsdeltaN}
E_{\delta,\varepsilon}(u) = \int_{\partial^*\{u=1\}}f_{\delta}(\frac{x}{\varepsilon},\nu_u(x))\mathrm{d}\mathcal{H}^{d-1}(x)
\end{align}
where for $\delta >0$
$f_{\delta} : \mathbb{R}^d \times \mathbb{R}^d \to [0,+\infty) $ is defined by
\begin{align*}
f_{\delta}(y,\nu)=\begin{cases}\sum_{\xi \in V} c_\xi^{k}|\langle\nu,\xi\rangle| &\text{if } y \in A_{k,\delta}, y \notin A_{j,\delta} \text{ for all } j \neq k , k=1,\cdots,N\\
\sum_{\xi \in V} \alpha_\xi |\langle \nu,\xi \rangle| &\text{if } y \in A_{k,\delta} \cap A_{j,\delta} \text{ for  some } j,k \in \{1,\cdots,N\}, j\neq k\\
\sum_{\xi \in V}c_\xi|\langle \nu,\xi\rangle| &\text{otherwise,}
\end{cases}
\end{align*} 
where $A_{j,\delta} = \{ y \in \mathbb{R}^d : \mathrm{dist}_{\infty}(y,\Pi_{\nu_j}+\mathbb{Z}^d)\leq \delta\}$. 

In fact $f_{\delta}$ increasingly converges to $f$ as $\delta \to 0$ on $\mathbb{R}^d \setminus \mathcal{N}$, where 
\begin{align*}
\mathcal{N} = \underset{j,k \in \{1,\cdots,N\}}{\bigcup_{j\neq k}}\left((\Pi_{\nu_k} +\mathbb{Z}^d ) \cap (\Pi_{\nu_j} +\mathbb{Z}^d )\right)
\end{align*}
where $\mathcal{H}^{d-1}(\mathcal{N})=0$. Hence again by the monotone convergence theorem and by [\cite{DM},Proposition 5.4] the claim follows. 

\medskip

\underline{Step 4:} Every functional of the form (\ref{EepsdeltaN}) can be approximated by a functional $E_{\eta,\delta,\varepsilon} : BV(\Omega,\{\pm 1\}) \to [0,+\infty)$ of the form 
\begin{align}\label{EepsdeltaNeta}
E_{\eta,\delta,\varepsilon}(u) = \int_{\partial^*\{u=1\}}f_{\eta,\delta}(\frac{x}{\varepsilon},\nu_u(x))\mathrm{d}\mathcal{H}^{d-1}(x)
\end{align}
where for $\eta,\delta >0$
$f_{\eta,\delta} : \mathbb{R}^d \times \mathbb{R}^d \to [0,+\infty) $ is defined by
\begin{align*}
f_{\eta,\delta}(y,\nu)=\begin{cases}\sum_{\xi \in V} c_\xi^{k}|\langle\nu,\xi\rangle| &\text{if } y \in A_{k,\delta}, y \notin A_{j,\delta} \text{ for all } j \neq k , k=1,\cdots,N\\
\sum_{\xi \in V} \alpha_\xi |\langle \nu,\xi \rangle| &\text{if } y \in A_{k,\delta} \cap A_{j,\delta} \text{ for  some } j,k \in \{1,\cdots,N\}, j\neq k\\
\sum_{\xi \in V} \beta_\xi |\langle \nu,\xi\rangle| &\text{if } y \in A_{k,\delta+\eta} \setminus A_{k,\delta}, y \notin A_{j,\delta} \text{ for all } j \neq k,k=1,\cdots,N  \\
\sum_{\xi \in V}c_\xi|\langle \nu,\xi\rangle| &\text{otherwise,}
\end{cases}
\end{align*}
 In fact $f_{\eta,\delta}$ decreasingly converges to $ f_{\delta}$ as $\eta \to 0$. Hence, by the monotone convergence theorem and by [\cite{DM},Proposition 5.4] the claim follows.

\medskip 
 
\underline{Step 5:} Every functional of the form (\ref{EepsdeltaNeta}) can be approximated by a functional $E_{n,\eta,\delta,\varepsilon,N} : BV(\Omega,\{\pm 1\}) \to [0,+\infty]$ of the form 
\begin{align}
E_{n,\eta,\delta,\varepsilon}(u)=\frac{1}{4}\sum_{\xi \in V} \sum_{i \in \Omega_\frac{1}{n}} \frac{1}{n^{d-1}} c_{i,\xi}^{n,\eta,\delta,\varepsilon}(u_{\frac{i}{n}}-u_{\frac{i+\xi}{n}})^2,
\end{align}
where $c^{n,\eta,\delta,\varepsilon}_{i,\xi}$ is $n$-periodic and 
\begin{align}\label{closeness}
|\theta(\{c^{n,\eta,\delta,\varepsilon}_{i,\xi}\})-\theta_\xi| \leq C(\varepsilon,N)\eta \text{ for all } \xi \in V
\end{align}
and such that $E_{n,\eta,\delta,\varepsilon}$ $\Gamma$-converges to $E_{\eta,\delta,\varepsilon}$ as $n \to \infty$. 

By Proposition \ref{Special} there exist $\{c_{i,\xi}^k\}$ $k=1,\cdots,N$, $c_{i,\xi}^\alpha,c_{i,\xi}^\beta,c_{i,\xi}^0$ $i \in \mathbb{Z}^d,\xi \in V$ $T$-periodic for some $T \in \mathbb{N}$ such that
\begin{align*}
\theta(\{c^{m}_{i,\xi}\})= \theta_\xi \text{ for all } m \in \{\alpha,0,1,\cdots,N\},\quad \theta(\{c_{i,\xi}^\beta\})=1
\end{align*}
and $E_n^m : BV(\Omega,\{\pm1\}) \to [0,+\infty]$, $m \in \{\alpha,\beta,0,1,\cdots,N\}$ defined by
\begin{align*}
E^m_n(u) = \frac{1}{4}\sum_{\xi \in V}\sum_{i \in \Omega_\frac{1}{n}}\frac{1}{n^{d-1}}c^m_{i,\xi} (u_{\frac{i}{n}}-u_{\frac{i+\xi}{n}})^2
\end{align*}
$\Gamma$-converges to $E^m : BV(\Omega,\{\pm 1\}) \to [0,+\infty)$, $m \in \{\alpha,\beta,0,1,\cdots,N\}$ defined by
\begin{align} \label{liminf*}
E^*(u) = \int_{\partial^*\{u=1\}} \sum_{\xi \in V}c^*_\xi|\langle \nu,\xi\rangle|\mathrm{d}\mathcal{H}^{d-1}
\end{align} 
where $c^\alpha_\xi = \alpha_\xi, c^\beta_\xi = \beta_\xi$ and $c^0_\xi = c_\xi$. We define $c^{n,\eta,\delta,\varepsilon}_{i,\xi}$ in $[-\frac{n}{2},\frac{n}{2})^d$ by
\begin{align} \label{cixi}
c^{n,\eta,\delta,\varepsilon}_{i,\xi}=\begin{cases}c_{i,\xi}^k &\text{if } i \in \left(A_{k,\delta}\right)_\frac{1}{n}, i \notin \left(A_{j,\delta}\right)_\frac{1}{n} \text{ for all } j \neq k , k=1,\cdots,N\\
c_{i,\xi}^\alpha&\text{if } i \in \left(A_{k,\delta} \cap A_{j,\delta}\right)_\frac{1}{n} \text{ for  some } j,k \in \{1,\cdots,N\}, j\neq k\\
c_{i,\xi}^\beta &\text{if } i \in \left(A_{k,\delta+\eta} \setminus A_{k,\delta}\right)_\frac{1}{n}, i \notin \left(A_{j,\delta}\right)_\frac{1}{n} \text{ for all } j \neq k,k=1,\cdots,N  \\
c_{i,\xi}^0&\text{otherwise,}
\end{cases}
\end{align}
and extend it $n$-periodically. Now (\ref{closeness}) holds, since 
\begin{align*}
\theta_\xi\leq\theta(\{c^{n,\eta,\delta,\varepsilon}_{i,\xi}\})&\leq \frac{1}{n^2}\Big( \#\left\{ i \in  \mathbb{Z}^d\cap\left( [-\frac{n}{2},\frac{n}{2})^d\setminus\left(A_{k,\delta+\eta} \setminus A_{k,\delta}\right)_\frac{1}{n} \right):c^{n,\eta,\delta,\varepsilon}_{i,\xi}=\beta_\xi\right\} \\& \hspace{7mm}+\# \left\{ i \in  \mathbb{Z}^d \cap \left( A_{k,\delta+\eta} \setminus A_{k,\delta}\right)_\frac{1}{n} :c^{n,\eta,\delta,\varepsilon}_{i,\xi}=\beta_\xi \right\}\Big)\\& \leq \theta_\xi +  C(d) \frac{N}{\varepsilon^d} \eta.
\end{align*}
It remains to show that $E_{n,\eta,\delta,\varepsilon} \, \Gamma$-converges to $E_{\eta,\delta,\varepsilon}$ as $n \to \infty$.

Set
\begin{align*}
E(u) =\Gamma\text{-}\lim_{n \to \infty}E_{n,\eta,\delta,\varepsilon}(u)
\end{align*}
which exists up to subsequences. By [\cite{AS},Theorem 4.2] we know that for all $(u,A) \in BV(\Omega,\{\pm 1\})\times \mathcal{A}(\Omega)$ we have that
\begin{align*}
E(u,A) = \int_{\partial^*\{u=1\} \cap A} \varphi'(x,\nu_u(x))\mathrm{d}\mathcal{H}^{d-1}(x)
\end{align*}
for some $\varphi' : \Omega \times \mathbb{R}^d \to [0,+\infty)$ where $E(\cdot,\cdot)$ is a localized version of $E(\cdot)$. Fix such $u \in BV(\Omega,\{\pm 1\})$. We are done if we prove $\varphi'(x,\nu) = f_{\eta,\delta}(\frac{x}{\varepsilon},\nu)$ for $\mathcal{H}^{d-1}$-a.e. $x \in \partial^*\{u=1\}$. We know that for $\mathcal{H}^{d-1}$-a.e. $x \in \partial^*\{u=1\}$ by the Radon-Nikodym Theorem there holds
\begin{align*}
\varphi'(x,\nu) = \lim_{\rho \to 0} \frac{1}{\rho^{d-1}}E(u, Q^\nu_\rho(x)).
\end{align*}
Now for $\rho >0$ fixed let $u_{n,\rho} \to u $ in $L^1( Q^\nu_\rho(x))$ as $n \to \infty$ be such that
\begin{align*}
\lim_{n \to \infty} E_{n,\eta,\delta,\varepsilon}(u_{n,\rho}, Q^\nu_\rho(x)) = \Gamma\text{-}\lim_{n \to \infty}E_{n,\eta,\delta,\varepsilon}(u, Q^\nu_\rho(x)).
\end{align*}
We obtain that
\begin{equation}\label{compphi}
\begin{split}
\varphi'(x,\nu) &= \lim_{\rho \to 0} \frac{1}{\rho^{d-1}}E(u, Q^\nu_\rho(x)) = \lim_{\rho \to 0} \frac{1}{\rho^{d-1}}\lim_{n \to \infty} E_{n,\eta,\delta,\varepsilon}(u_{n,\rho},Q^\nu_\rho(x))\\&=\lim_{\rho \to 0}\frac{1}{\rho^{d-1}}\lim_{n \to \infty} \frac{1}{4}\sum_{\xi \in V} \sum_{i \in \left( Q^\nu_\rho(x)\right)_\frac{1}{n}} \frac{1}{n^{d-1}} c_{i,\xi}^{n,\eta,\delta,\varepsilon}((u_{n,\rho})_{\frac{i}{n}}-(u_{n,\rho})_{\frac{i+\xi}{n}})^2.
\end{split}
\end{equation}
There are five cases to investigate
\begin{itemize}
\item[i)] $x \in \left(A_{k,\delta}\right)^\circ$, $x \notin A_{j,\delta}$ for all $ j\neq k $,for some $k=1,\cdots,N$
\item[ii)]$x \in A_{k,\delta} \cap A_{j,\delta}$ for some $j,k \in \{1,\cdots,N\}$ such that $j \neq k$
\item[iii)] $x \in \left(A_{k,\delta+\eta} \setminus A_{k,\delta}\right)^\circ$, $x \notin A_{j,\delta}$ for all $ j\neq k $,for some $k=1,\cdots,N$
\item[iv)] $x \in \partial A_{k,\delta}$ or  $x \in \partial A_{k,\delta+\eta}$ for some $k=1,\cdots,N$
\item[v)] otherwise
\end{itemize}
We only show (i) and (iv). The cases (i)-(iii),(v) are treated analogously.  

We first show (i), i.e.~let  $x\in \left(A_{k,\delta}\right)^\circ$ for some $k=1,\cdots,N$ and $x \notin A_{j,\delta}$ for all $j \neq k, j=1,\cdots,N$. By the definition (\ref{cixi}) we have that for $\rho $ small enough
\begin{align*}
E_{n,\eta,\delta,\varepsilon}(u_{n,\rho}, Q^\nu_\rho(x))=E^k_n(u_{n,\rho}, Q^\nu_\rho(x))
\end{align*}
and hence by (\ref{compphi}) and (\ref{liminf*}) we have that
\begin{align*}
\varphi'(x,\nu) &= \lim_{\rho \to 0} \frac{1}{\rho^{d-1}}\lim_{n \to \infty} E_{n,\eta,\delta,\varepsilon}(u_{n,\rho},Q^\nu_\rho(x))=\lim_{\rho \to 0} \frac{1}{\rho^{d-1}}\lim_{n \to \infty}E^k_n(u_{n,\rho},Q^\nu_\rho(x))\\& =\lim_{\rho \to 0} \frac{1}{\rho^{d-1}} E(u,  Q^\nu_\rho(x)) = f_{\eta,\delta}(\frac{x}{\varepsilon},\nu).
\end{align*}

Now we treat the case (iv), i.e.~either $x \in \partial A_{k,\delta}$ for some $k =1 ,\cdots,N$ or $x \in \partial A_{k,\delta+\eta}$ for some $k =1 ,\cdots,N$.   In the first  case we have that for $\rho$ small enough $ Q^\nu_\rho(x) \subset \left(A_{k,\delta +\eta}\right)^\circ \setminus \bigcup_{j \neq k}A_{k,\delta +\eta}$ and hence
\begin{align*}
\varphi'(x,\nu)&=\lim_{\rho \to 0} \frac{1}{\rho^{d-1}}\liminf_{n \to \infty} E_{n,\eta,\delta,\varepsilon}(u_{n,\rho},Q^\nu_\rho(x)) \geq \lim_{\rho \to 0} \frac{1}{\rho^{d-1}}\liminf_{n \to \infty}E^k_n(u_{n,\rho}, Q^\nu_\rho(x))\\&  = \sum_{\xi \in V} c^{k}_\xi |\langle\nu,\xi\rangle| = f_{\eta,\delta}(\frac{x}{\varepsilon},\nu).
\end{align*}
We now show that 
\begin{align}\label{boundaryineq}
\varphi'(x,\nu^k) \leq f_{\eta,\delta}\Big(\frac{x}{\varepsilon},\nu^k\Big) \text{ for } \mathcal{H}^{d-1} \text{ a.e. } x \in  \partial A_{k,\delta},
\end{align}
which is sufficient to prove that the $\Gamma$-limit agrees with $E_{\eta,\delta,\varepsilon}$, since $\nu_u = \nu^k$  $\mathcal{H}^{d-1}$-a.e. on $\partial A_{k,\delta}$.  Let 
\begin{align*}
\tilde{ d}:=\min_{\underset{j \neq k}{{1 \leq j \leq N}}}\mathrm{dist}_\infty(x, A_{j,\delta}).
\end{align*}
Fix $r<\frac{1}{2}\min\{\tilde{d},\delta\}$ small enough, assume that $\nu^k$ is the unit outer normal of the set $A_{k,\delta}$ and let $u \in BV_{loc}(\mathbb{R}^d;\{\pm 1\})$ be defined by
\begin{align}
u(z) = \chi_{Q^{\nu^k}_{\frac{r}{2}}(x)}\Big(z+ \nu^k\frac{r}{2}\Big)-2
\end{align}
Let $x_n \to x$ be such that $\{x_n\}_n \subset \left(A_{k,\delta}\right)^\circ$ and
\begin{align*}
\min_{\underset{j \neq k}{{1 \leq j \leq N}}}\mathrm{dist}_\infty(x_n,\partial A_{j,\delta}) \geq \frac{3}{4}\tilde{d}.
\end{align*}
Define $u_n \in BV_{loc}(\mathbb{R}^d;\{\pm 1\})$ by 
\begin{align}
u_n(x) = \chi_{Q^{\nu^k}_{\frac{r}{2}}(x_n)}\Big(z+ \nu^k\frac{r}{2}\Big)-2.
\end{align}

\begin{align}\label{valueu}
E(u) = \int_{\partial^*\{u=1\} \cap \partial A_{k,\delta}}\varphi'(y,\nu_u)\mathrm{d}\mathcal{H}^{d-1}(y)+ \int_{\partial^*\{u=1\}\setminus \partial A_{k,\delta}}\varphi'(y,\nu_u)\mathrm{d}\mathcal{H}^{d-1}(y)
\end{align}
and for all $n \in \mathbb{N}$ we have
\begin{align} \label{valueun}
E(u_n)&= \int_{\partial^*\{u=1\} \cap \{\nu_{u_n}=\nu_k\}}\varphi'(y,\nu_u)\mathrm{d}\mathcal{H}^{d-1}(y)+ \int_{\partial^*\{u=1\} \setminus \{\nu_{u_n}=\nu_k\}}\varphi'(y,\nu_u)\mathrm{d}\mathcal{H}^{d-1}(y)
\end{align}
Note that in (\ref{valueu}) and (\ref{valueun}) $\nu_u = \nu^k$  in both the first terms and that the second terms of (\ref{valueu}) and (\ref{valueun}) agree, because of (i). We  have $u_n \to u $ in $L^1_{loc}(\mathbb{R}^d;\{\pm 1\})$ and therefore by the lower semicontinuity of $E$ we have
\begin{align*} 
\liminf_{n \to \infty} E(u_n) \geq E(u)
\end{align*}
and therefore
\begin{align*}
r^{d-1}f_{\eta,\delta}\Big(\frac{x}{\varepsilon},\nu^k\Big)=r^{d-1}\liminf_{n \to \infty} f_{\eta,\delta}\Big(\frac{x_n}{\varepsilon},\nu^k\Big) \geq \int_{Q_{\frac{r}{2}}^{\nu^k}(x) \cap \partial A_{k,\delta}}\varphi'(y,\nu^k)\mathrm{d}\mathcal{H}^{d-1}(y)
\end{align*}

Dividing by $r^{d-1}$ and letting $r \to 0$, using the fact that $x$ is a Lebesgue point with respect to $\mathcal{H}^{d-1}\lfloor_{\partial A_{k,\delta} }$, the claim follows. The other case of (iv) can be done analogously and this therefore yields Step 5. 

\medskip

\underline{Step 6:} By the metrizability properties of $\Gamma$-convergence (see [\cite{DM},Theorem 10.22]),  Steps 1-5 together with a diagonal argument, noting that $\eta_{k},N,\varepsilon_k $ can be chosen such that for all $k \in \mathbb{N}$ we have that $\eta_{k}\frac{N}{\varepsilon_k^d}\leq \sqrt{\eta_{k}}$, yields that there exists a sequence of coefficients $c^{n_k}_{i,\xi}=c_{i,\xi}^{n_k,\eta_k,\delta_k,\varepsilon_k}$ $n_k$-periodic such that
\begin{align*}
\theta_\xi(\{c^{n_k}_{i,\xi}\}) \to \theta_\xi \text{ as } k \to \infty \text{ for all } \xi \in V
\end{align*} 
and $E_k : BV(\Omega,\{\pm 1\}) \to [0,+\infty]$ defined by 
\begin{align*}
E_{k}(u) = \frac{1}{4}\sum_{\xi \in V}\sum_{i \in \Omega_\frac{1}{n}} \frac{1}{n_k^{d-1}} c^{n_k}_{i,\xi}(u_{\frac{i}{n}} - u_{\frac{i+\xi}{n}})^2, \text{ for all } u \in PC_\frac{1}{n_k}(\Omega,\{\pm 1\})
\end{align*}
$\Gamma$-converges as $k \to \infty  $ to the functional $E_\varphi : BV(\Omega,\{\pm 1\}) \to [0,\infty)$ defined by
\begin{align*}
E_\varphi(u) = \int_{\partial^* \{u=1\}\cap \Omega} \varphi(\nu_u(x))\mathrm{d}\mathcal{H}^{d-1},
\end{align*}
which yields the claim.
\end{proof}
Now that we proved Theorem (\ref{Gammaclosure}) our goal is to show that this implies that the homogenized energy densities $\varphi^\varepsilon$ of the $c_{i,\xi}^\varepsilon$ converge to $\varphi$. This yields the next theorem and therefore implies that the bounds are optimal.
\begin{theorem}\label{General}
Let $\varphi : \mathbb{R}^d \to [0,+\infty)$ be convex, even, positively 1-homogeneous and such that
\begin{align}
\sum_{\xi\in V}\alpha_\xi|\langle \nu,\xi\rangle| \leq \varphi(\nu) \leq \sum_{\xi\in V}(\theta_\xi\beta_\xi+(1-\theta_\xi)\alpha_\xi)|\langle \nu,\xi\rangle|
\end{align}
with $\theta_\xi \in [0,1], \xi \in V$. If $\theta \in [0,1]$ satisfies (\ref{Averagefraction}) then $\varphi \in \textbf{H}_{\alpha,\beta,V}(\theta)$.
\end{theorem}
To prove Theorem \ref{General} we introduce the localization on regular open sets of $E_\varepsilon : BV(\Omega) \times \mathcal{A}^{reg}(\Omega) \to [0,+\infty]$ of the form
\begin{align}\label{Es}
E_\varepsilon(u,A) = \begin{cases}\frac{1}{4}\underset{\xi \in V}{\sum} \underset{i,i+\xi \in  A_\varepsilon}{\sum}\varepsilon^{d-1} c^\varepsilon_{i,\xi}(u_{\varepsilon i}-u_{\varepsilon(i+\xi)})^2 & \text{if } u \in \mathcal{PC}_\varepsilon(\Omega,\{\pm 1\})\\
+\infty &\text{otherwise,}
\end{cases}
\end{align}
and the localization on regular open sets of an auxiliary functional $F_\varepsilon : BV(\Omega) \times \mathcal{A}^{reg}(\Omega) \to [0,+\infty]$ defined by
\begin{align}\label{Fs}
F_\varepsilon(u,A) = \begin{cases}\frac{1}{2}\underset{\xi \in V}{\sum} \underset{i,i+\xi \in  A_\varepsilon}{\sum}\varepsilon^{d-1} c^\varepsilon_{i,\xi}|u_{\varepsilon i}-u_{\varepsilon(i+\xi)}| & \text{if } u \in \mathcal{PC}_\varepsilon(\Omega)\\
+\infty &\text{otherwise.}
\end{cases}
\end{align}
Note that $F_\varepsilon$ is the positively $1$-homogeneous extension of $E_\varepsilon$ to $\mathcal{PC}_\varepsilon(\Omega)$, that is to say that for $u \in \mathcal{PC}_\varepsilon(\Omega,\{\pm 1\})$ we have $\lambda E_\varepsilon(u,A) = \lambda F_\varepsilon(u,A) = F_\varepsilon(\lambda u,A)$.\\
\begin{remark}\label{F,E} \rm By \cite{AS} up to subsequences it holds that
\begin{align*}
\Gamma\text{-}\lim_{\varepsilon \to 0} E_\varepsilon(u,A) = \int_{\partial^*\{u=1\}\cap A} \varphi(x,\nu)\mathrm{d}\mathcal{H}^{d-1} =:E(u,A) 
\end{align*}
for all  $(u,A) \in BV(\Omega,\{\pm 1\}) \times \mathcal{A}(\Omega)$ and for some $\varphi :\Omega \times \mathbb{R}^d \to [0,+\infty)$ one homogeneous. Analogously as in \cite{AS}, using [\cite{BP},Remark 2.2, Lemma 4.2] it can be shown that up to the same subsequence
\begin{align*}
\Gamma\text{-}\lim_{\varepsilon \to 0} F_\varepsilon(u,A) = \frac{1}{2}\int_A \varphi\left(x,Du\right) =: F(u,A)
\end{align*}
for all $(u,A) \in BV(\Omega) \times \mathcal{A}(\Omega)$, where the energy densities of $E$ and $F$ agree. We have used the shorthand of
\begin{align*}
\int_A \varphi(x,Du) = \int_A \varphi\Big(x, \frac{dDu}{d|Du|}\Big)\mathrm{d}|Du|.
\end{align*}
Note that the functionals $F_\varepsilon$ satisfy suitable growth conditions, i.e. 
\begin{align*}
\frac{1}{C}|Du|(A) \leq F_\varepsilon(u,A) \leq C|Du|(A) 
\end{align*}
for all $(u,A) \in PC_{\varepsilon}(A) \times \mathcal{A}^{reg}(\Omega)$.
\end{remark}
The next Proposition establishes a cell formula for the $\Gamma$-limit of the auxiliary functional to recover the energy density, provided it is homogeneous in the spatial variable. Proposition \ref{Limsup minima} and Proposition \ref{Liminf minima} show the convergence of the cell formulas of the approximating (in the sense of $\Gamma$-convergence) energies to the cell formula of the limiting energy. Those three Propositions will then be used in the proof of Theorem \ref{General}.
 \begin{proposition}\label{CellFormula}
 Let $E : BV(\Omega) \times \mathcal{A}(\Omega) \to [0,+\infty] $ be defined by
 \begin{align}
 F(u,A) = \int_A \varphi(Du) 
 \end{align}
 for some $\varphi : \mathbb{R}^d \to [0,\infty)$, convex, positively $1$-homogeneous and such that
 \begin{align*}
\frac{1}{C}|\nu|\leq \varphi(\nu) \leq  C|\nu|
 \end{align*}
 for some $C >1$. Assume that $F(\cdot,A) $ is $L^1(A)$ lower semicontinuous for all $A \in \mathcal{A}(\Omega)$, then
 \begin{align}
 \varphi(\nu) = \inf \Big\{ \int_{[0,1)^d} \varphi(Du) : u \in BV_{loc}(\mathbb{R}^d) : u-\nu x \, 1\text{-periodic}\Big\}.
 \end{align}
 \end{proposition}
 \begin{proof}
 We prove 
 \begin{align}
 &\inf \Big\{ \int_{[0,1)^d} \varphi(Du) : u \in BV_{loc}(\mathbb{R}^d) : u-\nu x \, 1\text{-periodic}\Big\} \\ = &\inf \Big\{ \int_{Q} \varphi(Du) : u \in BV_{loc}(\mathbb{R}^d) : u-\nu x \, 1\text{-periodic}, |Du|(\partial Q)=0\Big\} \\= &\inf \Big\{ \int_{Q} \varphi(Du) : u \in W^{1,1}_{loc}(\mathbb{R}^d) : u-\nu x \, 1\text{-periodic}\Big\} =\varphi(\nu).
 \end{align}
 which yields the claim. Note that 
\begin{align*}
 &\inf \Big\{ \int_{[0,1)^d} \varphi(Du) : u \in BV_{loc}(\mathbb{R}^d) : u-\nu x \, 1\text{-periodic}\Big\} \\ \leq &\inf \Big\{ \int_{Q} \varphi(Du) : u \in BV_{loc}(\mathbb{R}^d) : u-\nu x \, 1\text{-periodic}, |Du|(\partial Q)=0\Big\} \\ \leq &\inf \Big\{ \int_{Q} \varphi(Du) : u \in W^{1,1}_{loc}(\mathbb{R}^d) : u-\nu x \, 1\text{-periodic}\Big\} \leq \varphi(\nu),
 \end{align*}  
 since we only decrease the set of admissible test functions in the minimum problems and in the last infimum $u(x) =\nu x$ is admissible. We prove
 \begin{equation}\label{Du=0}
 \begin{split}
  &\inf \Big\{ \int_{[0,1)^d} \varphi(Du) : u \in BV_{loc}(\mathbb{R}^d) : u-\nu x \, 1\text{-periodic}\Big\} \\ \geq &\inf \Big\{ \int_{Q} \varphi(Du) : u \in BV_{loc}(\mathbb{R}^d) : u-\nu x \, 1\text{-periodic}, |Du|(\partial Q)=0\Big\}.
  \end{split}
 \end{equation}
 To this end let $u \in BV_{loc}(\mathbb{R}^d)$ be such that $u-\nu x$ is $1\text{-periodic}$. Let $\tau \in \mathbb{R}^d$ be such that $u_\tau \in BV_{loc}(\mathbb{R}^d)$ defined by 
 \begin{align*}
 u_\tau(x) = u(x+\tau)
 \end{align*}
satisfies $|Du_\tau|(\partial Q) = 0$ and 
 \begin{align*}
 \int_{Q} \varphi(Du_\tau)= \int_{[0,1)^d} \varphi(Du_\tau) =\int_{[0,1)^d} \varphi(Du).
 \end{align*}
 Then this yields (\ref{Du=0}). Next we prove that
 \begin{equation}\label{W11}
 \begin{split}
  &\inf \Big\{ \int_{Q} \varphi(Du) : u \in BV_{loc}(\mathbb{R}^d) : u-\nu x \, 1\text{-periodic}, |Du|(\partial Q)=0\Big\} \\\geq &\inf \Big\{ \int_{Q} \varphi(Du) : u \in W^{1,1}_{loc}(\mathbb{R}^d) : u-\nu x \, 1\text{-periodic}\Big\}.
  \end{split}
 \end{equation}
 Let $ u \in BV_{loc}(\mathbb{R}^d)$ be such that $u-\nu x$ is $1\text{-periodic}, |Du|(\partial Q)=0$. Set $u_\varepsilon = u * \rho_\varepsilon$, where $\{\rho_\varepsilon\}_\varepsilon$ is a family of positive symmetric mollifiers. We have that $u_\varepsilon \to u$ in $L^1_{loc}(\mathbb{R}^d)$ and $|Du_\varepsilon|(Q) \to |Du|(Q)$ as $\varepsilon \to 0$. By the Reshetnyak Continuity Theorem [\cite{AFP},Theorem 2.39] we have that
 \begin{align*}
 \lim_{\varepsilon \to 0} \int_{Q}\varphi(Du_\varepsilon) = \int_{Q}\varphi(Du)
 \end{align*}
 and (\ref{W11}) follows. Now we prove 
 \begin{equation}
 \inf \Big\{ \int_{Q} \varphi(Du) : u \in W^{1,1}_{loc}(\mathbb{R}^d) : u-\nu x \, 1\text{-periodic}\Big\} \geq\varphi(\nu).
 \end{equation}
 Let $u \in W^{1,1}_{loc}(\mathbb{R}^d) $ be such that $u-\nu x$ is $1\text{-periodic}$. Let for $n \in \mathbb{N}$ $u_n \in W^{1,1}(Q)$ be defined by
\begin{align*}
u_n(x) = \frac{1}{n}u(nx).
\end{align*}
We then have that $u_n \to \nu x$ in $L^1(Q)$ and 
\begin{align} \label{equality}
\int_{Q}\varphi(Du_n)\mathrm{d}x = \int_{Q} \varphi(Du)\mathrm{d}x.
\end{align}
By the lower semicontinuity of $E(\cdot, Q)$ we obtain
\begin{equation}
\varphi(\nu) = F(\nu x ,Q) \leq \liminf_{n \to \infty} F(u_n, Q) = \int_{Q} \varphi(Du)\mathrm{d}x
\end{equation}
and the claim follows.
 \end{proof}
Set
\begin{align*}
 &m_\varepsilon(\nu) := \inf \Big\{ \frac{1}{2}\sum_{\xi \in V} \sum_{i \in \left([-\frac{1}{2},\frac{1}{2})^d\right)_\varepsilon}\varepsilon^{d-1} c_{i,\xi}^\varepsilon |u_{\varepsilon i} - u_{\varepsilon(i+\xi)}| : u \in \mathcal{PC}_{\varepsilon}(\mathbb{R}^d), u -\nu x \, 1 \text{-periodic}\Big\} \\
&m(\nu) := \inf \Big\{ F(u,[0,1)^d) : u \in BV_{loc}(\mathbb{R}^d) \, ,u -\nu x \, 1 \text{-periodic}\Big\}.
\end{align*}
 \begin{proposition} \label{Limsup minima} Let $F_\varepsilon : BV(\Omega) \times \mathcal{A}(\Omega) \to [0,\infty]$ be defined as in (\ref{Fs}) and let $F_\varepsilon$ $\Gamma$-converge with respect to the strong $L^1(\Omega)$ topology to the functional $F : BV(\Omega) \times \mathcal{A}(\Omega) \to [0,\infty]$ given by Remark \ref{F,E} and of the form as in Proposition \ref{CellFormula}. It then holds
\begin{equation}
\limsup_{\varepsilon \to 0} m_\varepsilon(\nu) \leq m(\nu).
\end{equation}
\end{proposition} 
\begin{proof}
 Fix $\eta >0$ and let $M := 2\sup_{\xi \in V} ||\xi||_\infty$, Let $u^\eta_\varepsilon \to \nu x$ be such that 
\begin{align} \label{nu approx}
\limsup_{\varepsilon\to 0} F_\varepsilon(u^\eta_\varepsilon,(1+\eta)Q) \leq F(\nu x,(1+\eta Q) \leq (1+\eta)^d \varphi(\nu).
\end{align} 
We now modify $u^\eta_\varepsilon$ such that it can be used as a test function for $m_\varepsilon(\nu)$. Set
\begin{align*}
\delta_\varepsilon = \int_Q |u_\varepsilon^\eta-\nu x|\mathrm{d}x
\end{align*}
and let $\varepsilon >0 $ be small enough and $k_\varepsilon \in \mathbb{N} $ be such that
\begin{align*}
\frac{\delta_\varepsilon}{\varepsilon}<<k_\varepsilon << \frac{\eta}{\varepsilon}
\end{align*}
and set $Q^i_{\eta,\varepsilon} = Q_{1-\eta-i \varepsilon M}$. Then we get
\begin{align*}
\delta_\varepsilon \geq \int_{Q \setminus Q_{(1-2\eta)}} |u_\varepsilon^\eta-\nu x|\mathrm{d}x\geq \sum_{i=0}^{k_\varepsilon-1}\int_{Q^i_{\eta,\varepsilon}\setminus Q^{i+1}_{\eta,\varepsilon} }|u_\varepsilon^\eta-\nu x|\mathrm{d}x,
\end{align*}
hence there exists $i_\varepsilon \in \{0,\cdots,k_\varepsilon\}$ such that
\begin{align}\label{smallness on the strip}
\sum_{i \in\left(Q^{i_\varepsilon}_{\eta,\varepsilon}\setminus Q^{i_\varepsilon+1}_{\eta,\varepsilon} \right)_\varepsilon } \varepsilon^d |(u_\varepsilon^\eta)_{\varepsilon i} - (\nu x)_{\varepsilon i}| \leq  C\int_{Q^{i_\varepsilon}_{\eta,\varepsilon}\setminus Q^{i_\varepsilon+1}_{\eta,\varepsilon} }|u_\varepsilon^\eta-\nu x|\mathrm{d}x << \frac{\delta_\varepsilon}{k_\varepsilon}<< \varepsilon,
\end{align}
 Now define $v^\eta_\varepsilon \in \mathcal{PC}_{\varepsilon}((1+\eta)Q)$ by
\begin{align*}
v_\varepsilon^\eta(i) =\begin{cases}
u^\eta_\varepsilon(i) & i \in Q^{i_\varepsilon}_{\eta,\varepsilon} \cap \varepsilon \mathbb{Z}^d \\
\nu i & \text{ otherwise on } \varepsilon \mathbb{Z}^d.
\end{cases}
\end{align*}
Note that, since $v^\eta_\varepsilon = \nu i $ on $(1+\eta)Q\setminus (1-\eta)Q$ it can be extended to the whole of $\mathbb{R}^d$ so that $v^\eta_\varepsilon  -\nu x$ is $1$-periodic. Thus we have 
\begin{align*}
m_\varepsilon(\nu) &\leq \sum_{\xi \in V} \sum_{i \in \left( [-\frac{1}{2},\frac{1}{2})^d\right)_\varepsilon} \varepsilon^{d-1} c_{i,\xi}^\varepsilon |(v_\varepsilon^\eta)_{\varepsilon i} - (v_\varepsilon^\eta)_{\varepsilon (i+\xi)} | \leq F_\varepsilon(v_\varepsilon^\eta,(1+\eta)Q) \\&\leq F_\varepsilon(u_\varepsilon^\eta,(1+\eta)Q) +F_\varepsilon(\nu x,(1+\eta)Q\setminus (1-2\eta)Q) \\&+ \sum_{\xi \in V} \sum_{i \in \left(Q^{i_\varepsilon}_{\eta,\varepsilon}\setminus Q^{i_\varepsilon+1}_{\eta,\varepsilon} \right)_\varepsilon} \varepsilon^{d-1} c_{i,\xi}^\varepsilon \left(|(u)_{\varepsilon i} - (\nu x)_{\varepsilon (i +\xi)}| + |(u)_{\varepsilon (i+\xi)} - (\nu x)_{\varepsilon  i}|\right).
\end{align*}
Noting that 
\begin{align*}
|u_{\varepsilon (i+\xi)} - (v )_{\varepsilon  i}| \leq |(u)_{\varepsilon (i+\xi)} - (v)_{\varepsilon  (i+\xi)}| +|(v)_{\varepsilon (i+\xi)} - (v)_{\varepsilon  i}|,
\end{align*}
 using (\ref{smallness on the strip}) and the growth conditions in Remark \ref{F,E} we obtain
 \begin{align*}
 m_\varepsilon(\nu) \leq F_\varepsilon(u_\varepsilon^\eta,(1+\eta)Q) +C \eta^d +o(1).
 \end{align*}
 Therefore by (\ref{nu approx}) we obtain
 \begin{align*}
 \limsup_{\varepsilon \to 0} m_\varepsilon(\nu) \leq (1+\eta)^d \varphi(\nu) + C\eta^d.
 \end{align*}
The claim follows by letting $\eta \to 0$.

\end{proof}

\begin{proposition} \label{Liminf minima} Let $F_\varepsilon : BV(\Omega) \times \mathcal{A}(\Omega) \to [0,\infty]$ be defined as in (\ref{Fs}) with $c^{\varepsilon}_{i,\xi}$ $\frac{1}{\varepsilon}$-periodic and let $F_\varepsilon$ $\Gamma$-converge with respect to the strong $L^1(\Omega)$-topology to the functional $F : BV(\Omega) \times \mathcal{A}(\Omega) \to [0,\infty]$ given by Remark \ref{F,E} and of the form as in Proposition \ref{CellFormula}. It then holds
\begin{equation}
\liminf_{\varepsilon \to 0} m_\varepsilon(\nu) \geq m(\nu).
\end{equation}
\end{proposition} 
\begin{proof} Let $u^\varepsilon \in \mathcal{PC}_{\varepsilon}(\mathbb{R}^d)$ be such that $u_\varepsilon-\nu x$ is $1$-periodic, $\fint_Q u_\varepsilon =0$ and
\begin{align}\label{bound me}
\sum_{\xi \in V} \sum_{i \in \left( [-\frac{1}{2},\frac{1}{2})^d\right)_\varepsilon} \varepsilon^{d-1} c^\varepsilon_{i,\xi}|(u^\varepsilon)_{\varepsilon i } -(u^\varepsilon)_{\varepsilon(i+\xi)}| \leq m_\varepsilon(\nu) +\varepsilon.
\end{align}
Since $u^\varepsilon -\nu x$ is $1$-periodic, $c^\varepsilon_{i,\xi}$ are $\frac{1}{\varepsilon}$-periodic and by the growth condition of $F_\varepsilon$ we have that 
\begin{align*}
\sup_{\varepsilon >0} |Du^\varepsilon|(B_R) \leq C_R <+\infty
\end{align*}
Hence by Poincar\'e Inequality we have that 
\begin{align*}
\sup_{\varepsilon > 0 } ||u^\varepsilon||_{BV(B_R)} <\infty,
\end{align*}
and therefore up to subsequences we have that $u^\varepsilon \to u $ in $L^1_{loc}(\mathbb{R}^d)$, $u \in BV_{loc}(\mathbb{R}^d)$ and $u-\nu x$ is $1$-periodic. In order to use $u$ as a test function for $m$, which can be compared to $m_\varepsilon$ it is necessary to translate it so that it does not concentrate energy on the boundary. Choose $x_0 \in \mathbb{R}^d$, such that $|Du|(\partial Q(x_0))=0$ and by the $1$-periodicity of $u_\varepsilon -\nu x, u-\nu x$, $\frac{1}{\varepsilon}$-periodicity of $c_{i,\xi}^\varepsilon$ respectively we have that
\begin{align*}
&\frac{1}{2}\sum_{\xi \in V} \sum_{i \in \left([-\frac{1}{2},\frac{1}{2})^d\right)_\varepsilon}\varepsilon^{d-1} c_{i,\xi}^\varepsilon |(u^\varepsilon)_{\varepsilon i} - (u^\varepsilon)_{\varepsilon(i+\xi)}|=\frac{1}{2}\sum_{\xi \in V} \sum_{i \in \left(x_0+[-\frac{1}{2},\frac{1}{2})^d\right)_\varepsilon}\varepsilon^{d-1} c_{i,\xi}^\varepsilon |(u^\varepsilon)_{\varepsilon i} - (u^\varepsilon)_{\varepsilon(i+\xi)}| 
\end{align*}
and
\begin{align*}
\int_{\left[-\frac{1}{2},\frac{1}{2}\right)^d} \varphi(Du) = \int_{x_0 + \left[-\frac{1}{2},\frac{1}{2}\right)^d} \varphi(Du) = \int_{Q(x_0)} \varphi(Du).
\end{align*} 
Using (\ref{bound me}) and using that $F_\varepsilon$  $\Gamma$-converges to  $F$ we have that
\begin{align*}
\liminf_{\varepsilon \to 0} m_\varepsilon(\nu) &\geq \liminf_{\varepsilon \to 0} \frac{1}{2}\sum_{\xi \in V} \sum_{i \in \left([-\frac{1}{2},\frac{1}{2})^d\right)_\varepsilon}\varepsilon^{d-1} c_{i,\xi}^\varepsilon |(u^\varepsilon)_{\varepsilon i} - (u^\varepsilon)_{\varepsilon(i+\xi)}| \\&= \liminf_{\varepsilon \to 0}\frac{1}{2}\sum_{\xi \in V} \sum_{i \in \left(x_0+[-\frac{1}{2},\frac{1}{2})^d\right)_\varepsilon}\varepsilon^{d-1} c_{i,\xi}^\varepsilon |(u^\varepsilon)_{\varepsilon i} - (u^\varepsilon)_{\varepsilon(i+\xi)}|\\&\geq \liminf_{\varepsilon \to 0} F_\varepsilon(u^\varepsilon,Q(x_0))\geq F(u,Q(x_0)) = \int_{[-\frac{1}{2},\frac{1}{2})^d} \varphi(Du)\geq m(\nu)
\end{align*}
and the claim follows.
\end{proof}
\begin{proof}[Proof of Theorem \ref{General}:] 
 Let $ \theta \in [0,1], \varphi$ be given as in Theorem \ref{General}. By Theorem \ref{Gammaclosure}, we know that there exist a sequence $\varepsilon \to 0$ and $\{c^\varepsilon_{ij}\}_\varepsilon$ $\frac{1}{\varepsilon}$-periodic, $\theta(\{c^\varepsilon_{i,\xi}\}) \to \theta$ such that if we define $E_{\eta,\varepsilon} : P_{\eta  \varepsilon}(\Omega, \{\pm 1\}) \times \mathcal{A}(\Omega) \to [0,+\infty]$ by 
\begin{align*}
E_{\eta ,\varepsilon}(u,A) = \frac{1}{4}\sum_{\xi \in V} \sum_{ i,i+\xi \in A_{\eta\varepsilon}}(\eta\varepsilon)^{d-1} c^\varepsilon_{i,\xi}(u_{\eta\varepsilon i}-u_{\eta\varepsilon(i+\xi)} )^2 
\end{align*}
we have that $E_{1,\varepsilon} \overset{\Gamma}{\to} E$, where $E: BV(\Omega,\{\pm 1\}) \to [0,+\infty]$ is defined by 
\begin{align*}
E(u,A) = \int_{\partial^*\{u=1\} \cap A} \varphi(\nu) \mathrm{d}\mathcal{H}^{d-1}.
\end{align*}
Introducing the auxiliary functionals $F_{\eta, \varepsilon} : P_{\eta \varepsilon}(\Omega) \times \mathcal{A}(\Omega) \to [0,+\infty)$ as in (\ref{Fs}) by [\cite{AS},Thm 4.7],  up to subsequences it holds that 
\begin{align}
\Gamma \text{-} \lim_{\eta \to 0} E_{\eta, \varepsilon}(u,A) = \int_{\partial^*\{u=1\} \cap A } \varphi_\varepsilon(\nu)\mathrm{d}\mathcal{H}^{d-1} =: E_\varepsilon(u,A) \label{Varphi e}
\end{align}
and by Remark \ref{F,E} 
\begin{align}
&\Gamma \text{-} \lim_{\varepsilon \to 0} F_{1, \varepsilon}(u,A) = \frac{1}{2}\int_A \varphi(Du) =: F(u,A),\label{Varphi F} \\
&\Gamma \text{-} \lim_{\eta \to 0} F_{\eta, \varepsilon}(u,A) = \frac{1}{2}\int_A \varphi_\varepsilon(Du) =: F_\varepsilon(u,A).
\end{align}
The normalization factor of $\frac{1}{2}$ appears in order that the functionals $E$ and $F$ agree on functions in $BV(\Omega,\{\pm 1\})$ (as mentioned in Remark \ref{F,E}).
By (\ref{Varphi e}), noting that the period of the coefficients is fixed with fixed $\varepsilon$, we have that $\varphi_\varepsilon \in \textbf{H}_{\alpha,\beta,V}(\theta)$. We have that
\begin{equation}\label{SingleCell}
\begin{split}
&\inf\Big\{ \sum_{\xi \in V} \sum_{i \in \ [-\frac{1}{2\varepsilon\eta},\frac{1}{2\varepsilon\eta})^d \cap \mathbb{Z}^d }(\varepsilon\eta)^{d-1}c_{i,\xi}^\varepsilon |u_{\varepsilon \eta i} - u_{\varepsilon \eta (i+\xi)}| : u \in \mathcal{PC}_{\varepsilon\eta}(\mathbb{R}^d), u-\nu x 1\text{-periodic}\Big\}\\&=\inf\Big\{ \sum_{\xi \in V} \sum_{i \in \left( [-\frac{1}{2},\frac{1}{2})^d\right)_{\varepsilon} }\varepsilon^{d-1}c_{i,\xi}^\varepsilon |u_{\varepsilon i} - u_{\varepsilon  (i+\xi)}| : u \in \mathcal{PC}_{\varepsilon}(\mathbb{R}^d), u-\nu x 1\text{-periodic}\Big\}
\end{split}
\end{equation}
In fact for every $u \in \mathcal{PC}_{\varepsilon\eta}(\mathbb{R}^d)$ with $u-\nu x$ is $ 1$-periodic we can define $\tilde{u} \in \mathcal{PC}_{\varepsilon}(\mathbb{R}^d)$ , $\tilde{u}-\nu x$  $1$-periodic by setting
\begin{align*}
\tilde{u}(z)= \eta^{d-1}\sum_{i \in \left( [-\frac{1}{2},\frac{1}{2})^d\right)_\eta} u(\eta (z + i)) 
\end{align*} 
such that by convexity, positive 1-homogeneity and the periodicity of the $c_{i,\xi}^\varepsilon$ 
\begin{align*}
\sum_{\xi \in V} \sum_{i \in \left( [-\frac{1}{2},\frac{1}{2})^d\right)_{\varepsilon\eta} }(\varepsilon\eta)^{d-1}c_{i,\xi}^\varepsilon |u_{\varepsilon \eta i} - u_{\varepsilon \eta (i+\xi)}| \geq   \sum_{\xi \in V} \sum_{i \in \left( [-\frac{1}{2},\frac{1}{2})^d\right)_{\varepsilon} }\varepsilon^{d-1}c_{i,\xi}^\varepsilon |\tilde{u}_{\varepsilon i} - \tilde{u}_{\varepsilon  (i+\xi)}|
\end{align*}
holds.
On the other hand for every $u \in \mathcal{PC}_{\varepsilon}(\mathbb{R}^d)$ ,$u-\nu x \, 1$-periodic one defines $\tilde{u} \in \mathcal{PC}_{\varepsilon \eta}(\mathbb{R}^d)$, $\tilde{u}-\nu x\,1$-periodic by
\begin{align*}
\tilde{u}(z)= \eta \, u(\frac{z}{\eta})
\end{align*}
for which by convexity, positive 1-homogeneity and the periodicity of the $c_{i,\xi}^\varepsilon$ 
\begin{align*}
\sum_{\xi \in V} \sum_{i \in \left( [-\frac{1}{2},\frac{1}{2})^d\right)_{\varepsilon\eta} }(\varepsilon\eta)^{d-1}c_{i,\xi}^\varepsilon |\tilde{u}_{\varepsilon \eta i} - \tilde{u}_{\varepsilon \eta (i+\xi)}| \leq   \sum_{\xi \in V} \sum_{i \in \left( [-\frac{1}{2},\frac{1}{2})^d\right)_{\varepsilon} }\varepsilon^{d-1}c_{i,\xi}^\varepsilon |u_{\varepsilon i} - u_{\varepsilon  (i+\xi)}|
\end{align*}
holds.
By (\ref{Limsup minima}), (\ref{Liminf minima}),(\ref{SingleCell}) and (\ref{Varphi F}) we have that for all $\nu \in S^{d-1}$ it holds
\begin{align*}
&\lim_{\varepsilon \to 0}\frac{1}{2}\varphi_\varepsilon(\nu)=\lim_{\varepsilon \to 0} m_\varepsilon(\nu)= m(\nu) = \frac{1}{2}\varphi(\nu)
\end{align*}
and therefore $\varphi \in \textbf{H}_{\alpha,\beta,V}(\theta)$.
\end{proof}
\section{Localization Principle}
The goal of this section is the computation of the \textit{G-closure} of mixtures, i.e.~all possible limits of mixtures where the interaction coefficients $\{c_{i,\xi}\}$ need not be periodic anymore. We show a \textit{localization principle}, which says that this computation can be reduced to the optimal bounds of periodic mixtures. We state this in the two main theorems below.

\begin{remark} \rm

We deal with surface energies $E : BV(\Omega; \{ \pm 1\}) \times \mathcal{A}(\Omega) \to [0,\infty)$ of the form 
\begin{align}
E(u,A) = \int_{\partial^*\{u=1\}\cap A} g(x,\nu_u(x))\mathrm{d}\mathcal{H}^{d-1},
\end{align}
where $g : \Omega \times \mathbb{R}^d \to [0,\infty)$ satisfies
\begin{align}
\frac{1}{C}|\nu| \leq g(x,\nu) \leq C|\nu| 
\end{align}
 for all  $(x,\nu) \in \Omega \times \mathbb{R}^d$ and $E(\cdot,A) $ is $L^1(A)$-lower semicontinuous. 

Let $\varphi : \Omega \times \mathbb{R}^d\to[0,+\infty)$ defined by
\begin{align} \label{blowup}
\varphi(x,\nu) = \limsup_{\rho \to 0} \frac{m(x,\nu,\rho)}{w_{d-1}r^{d-1}},
\end{align}
where $m(x,\nu,\rho) : \Omega \times \mathbb{R}^d \times (0,\mathrm{dist}(x,\partial \Omega)) \to [0,\infty) $ is the one homogeneous extension in the second variable of
\begin{equation}\label{minimumproblem}
\begin{split}
m(x,\nu,\rho) =  \inf \{ E&(v,B_\rho(x)) : v \in BV(\Omega; \{ \pm 1\}) \\& v = u_{x,\nu} \text{ in a neighborhood of } \partial B_\rho(x)\}.
\end{split}
\end{equation}
By \cite{BFM} we have that
\begin{align}
E(u,A) = \int_{\partial^*\{u=1\}\cap A} \varphi(x,\nu_u(x))\mathrm{d}\mathcal{H}^{d-1}
\end{align}
for all $(u,A) \in BV(\Omega;\{\pm 1\}) \times \mathcal{A}(\Omega)$, so that $\varphi$ is equivalent for $g$ in our considerations. 
We will therefore establish the \textit{localization principle} for $\varphi$.
\end{remark}

For $c_{i,\xi}^\varepsilon \in \{\alpha_{\xi},\beta_\xi\}^{\Omega_\varepsilon}$ we define (with abuse of notation) the {\em local volume fraction of $\beta$-bonds}  by
\begin{align}\label{localVolumefractions}
\theta_\xi(\{ c^\varepsilon_{i,\xi} \}) = \underset{c^\varepsilon_{i,\xi}=\beta_\xi}{\sum_{i\in \Omega_\varepsilon}} \varepsilon^d \delta_{\varepsilon i}, \quad \theta(\{ c^\varepsilon_{i,\xi} \}) = \frac{1}{\# V}\sum_{\xi \in V}\underset{c^\varepsilon_{i,\xi}=\beta_\xi}{\sum_{i\in \Omega_\varepsilon}} \varepsilon^d \delta_{\varepsilon i}.
\end{align}
\begin{theorem}\label{Localization 2} Let $\{c^\varepsilon_{i,\xi}\}_\varepsilon\in \{\alpha_\xi,\beta_\xi\}^{\Omega_\varepsilon}$ and let 
\begin{align}
E_\varepsilon (u) = \frac{1}{4}\sum_{\xi \in V}\sum_{i,i+\xi \in \Omega_\varepsilon} \varepsilon^{d-1} c^\varepsilon_{i,\xi} (u_{\varepsilon i}-u_{\varepsilon(i+\xi)})^2.
\end{align}
Assume that $ \theta(\{c_{i,\xi}^\varepsilon\}) \overset{*}{\rightharpoonup} \theta$ and $E_\varepsilon \,\, \Gamma$-converges to $E:BV(\Omega;\{\pm 1\}) \to [0,\infty)$ given by
\begin{align*}
E(u) = \int_{\partial^*\{u=1\}\cap \Omega} \varphi(x,\nu_u(x))\mathrm{d}\mathcal{H}^{d-1}
\end{align*}
with $\varphi$ satisfying (\ref{blowup}). Then $\varphi(x,\cdot) \in \textbf{H}_{\alpha,\beta,V}(\theta(x))$ for a.e. $x \in \Omega$.

\end{theorem}
\begin{theorem}\label{Localization 1} Let $\theta : \Omega \to [0,1]$ be measurable and $\varphi : \Omega \times \mathbb{R}^d \to [0,+\infty) $ be positively $1$-homogeneous in the second variable such that the trivial bounds (\ref{tri}) are satisfied and $\varphi(x,\cdot) \in \textbf{H}_{\alpha,\beta,V}(\theta(x))$ for a.e. $x \in \Omega$. Then there exist $\{c^\varepsilon_{i,\xi}\}_\varepsilon \in \{\alpha_\xi,\beta_\xi\}^{\Omega_\varepsilon} $ such that $E_\varepsilon$ $\Gamma$-converges to $E$ where
\begin{equation*}
E_\varepsilon(u) = \frac{1}{4}\sum_{\xi \in V}\sum_{i,i+\xi \in \Omega_\varepsilon} \varepsilon^{d-1} c^\varepsilon_{i,\xi} (u_{\varepsilon i}-u_{\varepsilon(i+\xi)})^2, \quad E(u) = \int_{\partial^*\{u=1\} \cap \Omega} \varphi(x,\nu_u(x))\mathrm{d}\mathcal{H}^{d-1}
\end{equation*}
and
$ \displaystyle
\theta(\{ c^\varepsilon_{i,\xi} \}) \overset{*}{\rightharpoonup} \theta
$
as $\varepsilon \to 0$.
\end{theorem} 

Theorem \ref{Localization 2} establishes the fact that at a.e.~$x \in \Omega$ we can reduce to the periodic setting and Theorem \ref{Localization 1} establishes the optimality of this condition, i.e.~every surface energy whose energy density satisfies for a.e.~$x \in \Omega$ that $\varphi(x,\cdot) \in \textbf{H}_{\alpha,\beta,V}(\theta(x))$ for some measurable function $0\leq \theta\leq 1$ can be recovered as the $\Gamma$-limit of some discrete energies of the form (\ref{gamma}), whose local volume fractions $\theta(\{c_{i,\xi}^\varepsilon\})$ of $\beta$-bonds converge (weakly*) to the limiting volume fraction $\theta$. Note that the assumptions of Theorem \ref{Localization 2} are always satisfied, up to a subsequence.

\begin{proof}[Proof of Theorem \ref{Localization 2}:]
By Remark \ref{F,E} and  [\cite{BFM},Remark 3.8] we have that $\varphi(x,\cdot) $ is convex for $a.e. \, x \in \Omega$.  $\varphi(x,\cdot)$ is even and positively $1$-homogeneous. Thus we are done if we show 
\begin{align} \label{Bounds}
\sum_{\xi \in V} \alpha_\xi |\langle \nu, \xi\rangle|\leq \varphi(x,\nu)\leq \sum_{\xi \in V} (\theta_\xi(x) \beta_\xi + (1-\theta_\xi(x))\alpha_\xi )|\langle \nu,\xi\rangle|
\end{align}
for a.e. $x \in \Omega$ and $\theta$ the weak*-limit of $\theta(\{c_{i,\xi}^\varepsilon\})$ satisfying (\ref{Averagefraction}). First of all note that the lower bound in (\ref{Bounds}) is trivial, noting that $c^\varepsilon_{i,\xi} \geq \alpha_\xi$. 
We have that $\theta(\{c_{i,\xi}^\varepsilon\}) \overset{*}{\rightharpoonup} \theta_\xi$, $\theta_\xi \in [0,1]$ and $ \frac{1}{\# V}\sum_{\xi \in V}\theta_\xi = \theta$.

 We prove the estimate for all points in $E$ where 
\begin{align*}
E:= \{ x \in \Omega : \varphi(x,\cdot) \text{ is convex and } x \text{ is a Lebesgue point for } \theta_\xi,\xi \in V\}.
\end{align*}
Since $\varphi(x,\cdot)$ is convex it suffices to prove for all $\nu = \frac{v}{||v||}$, $v \in V$
\begin{align}\label{boundpsi}
\varphi\left(x,\nu\right) \leq \sum_{\xi \in V}(\theta_\xi(x)\beta_\xi + (1-\theta_\xi(x))\alpha_\xi)|\langle \nu,\xi\rangle|=: \psi\left(x, \nu\right),
\end{align}
since
\begin{align*}
\psi(x,\nu)= \sup \{ g : g \text{ even, convex, pos. } 1 \text{-homogenous } g(v) \leq \psi\left(x,v\right) \text{ for all } v \in V \}.  
\end{align*}
We know that for all $x \in E$ and for all $\delta >0$, $\xi,v \in V$ we have
\begin{align*}
\lim_{\rho \to 0} \fint_{B^v_{\rho,\delta}(x)} |\theta_\xi(y)-\theta_\xi(x)|\mathrm{d}y = 0,
\end{align*}
with $B^v_{\rho,\delta}(x) = B_\rho(x)\cap \{ y \in \mathbb{R}^d : |\langle y-x,\frac{v}{||v||}\rangle| \leq \rho \delta\}  $. 

We now prove (\ref{boundpsi}) for $v \in V$.  To this end we construct a suitable competitor in the minimum problem for $m_\varepsilon(x,\frac{v}{||v||},\rho)$. Note that by [\cite{AS},Theorem 4.9] for suitable $\rho \to 0$ we have $m_\varepsilon(x,\frac{v}{||v||},\rho) \to m(x,\frac{v}{||v||},\rho)$ as $\varepsilon \to 0$. Let $u^{k,\rho}_\varepsilon \in BV(B_\rho,\{\pm 1\}), k \in (B^v_{\rho,\delta}(x))_\varepsilon$ be defined by
\begin{align*}
u^{k,\rho}_\varepsilon(z) = \begin{cases} u_{\varepsilon k ,\frac{v}{||v||}}(z) & \mathrm{dist}(z,\partial B_\rho(x)) >\rho \delta   \cap \varepsilon\mathbb{Z}^d \\
u_{x,\frac{v}{||v||}}(z) &\text{otherwise for } z\in\varepsilon\mathbb{Z}^d.
\end{cases}
\end{align*}
Estimating the energy yields
\begin{align*}
E_\varepsilon(u^{k,\rho}_\varepsilon,B_\rho(x)) \leq \sum_{\xi \in V}\sum_{i \in \Pi_{\frac{v}{||v||}}^\xi(k) \cap (B_\rho(x))_\varepsilon  }  \varepsilon^{d-1} c^\varepsilon_{i,\xi} +O(\delta \rho^{d-1}),
\end{align*}
where $O(\delta \rho^{d-1})$ is the contribution due to the boundary. Hence there exists $k_0=k^{\rho,\delta}_0 \in (B^v_{\rho,\delta}(x))_\varepsilon$ such that 
\begin{align*}
E_\varepsilon(u^{k_0,\rho}_\varepsilon,B_\rho(x))  \leq \frac{1}{\# B^v_{\rho,\delta}(x))_\varepsilon } \sum_{k \in B^v_{\rho,\delta}(x))_\varepsilon } E_\varepsilon(u^{k,\rho}_\varepsilon,B_\rho(x)).
\end{align*}
Hence
\begin{align*}
m_\varepsilon(x,\frac{v}{||v||},\rho) &\leq E_\varepsilon(u_\varepsilon^{k_0,\rho},B_\rho(x))\\&\leq \frac{1}{\# (B^v_{\rho,\delta}(x))_\varepsilon}\sum_{\xi \in V} \sum_{k \in (B^v_{\rho,\delta}(x))_\varepsilon} \sum_{i \in \Pi_{\frac{v}{||v||}}^\xi(k) \cap (B_\rho)_\varepsilon} \varepsilon^{d-1} c_{i,\xi}^\varepsilon + O(\delta \rho^{d-1}) \\& \leq \frac{\varepsilon^d}{|B^v_{\rho,\delta}(x)|}\sum_{\xi \in V} \sum_{i \in (B_\rho)_\varepsilon} \varepsilon^{d-1}c^\varepsilon_{i,\xi} \#\{ k \in (B^v_{\rho,\delta}(x))_\varepsilon  : i \in \Pi_{\frac{v}{||v||}}^\xi(k)\} +O(\delta \rho^{d-1}) \\& \leq
\frac{\varepsilon^d}{|B^v_{\rho,\delta}(x)|}\sum_{\xi \in V}\sum_{i \in (B^v_{\rho,\delta}(x))_\varepsilon} \varepsilon^{d-1}c_{i,\xi}^\varepsilon |\langle \frac{v}{||v||},\xi\rangle| w_{d-1} \frac{\rho^{d-1}}{\varepsilon^{d-1}}  +O(\delta\rho^{d-1})\\&
= \frac{w_{d-1} \rho^{d-1}}{|B^v_{\rho,\delta}(x)|}\sum_{\xi \in V}( \theta^\varepsilon_\xi(B^v_{\rho,\delta}(x))\beta_\xi + (|B^v_{\rho,\delta}(x)|- \theta_\xi^\varepsilon(B^v_{\rho,\delta}(x))\alpha_\xi)|\langle \frac{v}{||v||},\xi\rangle| +O(\delta \rho^{d-1})
\end{align*}
Dividing by $w_{d-1}\rho^{d-1} $, $x \in E$, taking the limit as $\varepsilon \to 0$, $\limsup$ as $\rho \to 0$ and using the weak convergence of measures, together with the fact that $\theta_\xi(\partial B^v_{\rho,\delta}(x))=0$  we obtain that
\begin{align*}
\varphi(x,\frac{v}{||v||}) &\leq \sum_{\xi \in V}\limsup_{\rho \to 0 }\left(\fint_{B^v_{\rho,\delta}(x)}\theta_\xi(y) \mathrm{d}y\beta_\xi +  (1-\fint_{B^v_{\rho,\delta}(x)}\theta_\xi(y) \mathrm{d}y)\alpha_\xi\right)|\langle \frac{v}{||v||},\xi\rangle| + O(\delta)\\&=\sum_{\xi \in V} (\theta_\xi(x) \beta_\xi + (1-\theta_\xi(x))\alpha_\xi )|\langle \frac{v}{||v||},\xi\rangle| + O(\delta)
\end{align*}
The claim follows by letting $\delta \to 0$.
\end{proof}

We need first to establish some properties of $m$ defined in (\ref{minimumproblem}).

\begin{proposition}\label{Propertiesm} The following holds:

 \begin{itemize}
 \item[i)] For all $x \in \Omega , \nu \in S^{d-1}$ we have that $ \rho \mapsto m(x,\nu, \rho) $ is continuous on $(0,\mathrm{dist}(x,\partial \Omega)) \setminus E(x,\nu)$ where $E(x,\nu) \subset (0,\mathrm{dist}(x,\partial \Omega))$ is countable.
 \item[ii)] For all $x \in \Omega, \nu_1,\nu_2 \in S^{d-1}, \rho \in (0,\mathrm{dist}(x,\partial \Omega))$ there exists a modulus of continuity $w:[0,\infty) \to [0,\infty)$ such that
 \begin{align*}
 |m(x,\nu_1,\rho)-m(x,\nu_2,\rho)| \leq \rho^{d-1}w(|\nu_1-\nu_2|).
 \end{align*}
 \item[iii)] Let $x_0 \in \Omega, \rho_0 \in (0,\mathrm{dist}(x_0,\partial \Omega)), \nu \in S^{d-1}$ and assume that $ \rho \mapsto m(x_0,\nu,\rho)$ is continuous at $\rho_0 $, then $x \mapsto m(x,\nu,\rho_0) $ is continuous at $x_0$.
\end{itemize}
\end{proposition}
 \begin{proof}
(i) Fix $x \in \Omega,\nu \in S^{d-1}$, set $r:= \mathrm{dist}(x,\partial \Omega)$ and define $ m : (0,d) \to [0,+\infty)$ by
 \begin{align*}
 m(\rho) = m(x,\nu,\rho) + \int_{\Pi_\nu(x) \cap\left( B_r(x) \setminus B_\rho(x)\right)} g(y,\nu(y)) \mathrm{d}\mathcal{H}^{d-1}.
 \end{align*}
\underline{Claim:} For $ 0 < \rho_1 < \rho_2 < r $ it holds that
 \begin{align*}
 m(\rho_2) \leq m(\rho_1).
 \end{align*}
 \underline{Proof of the claim:} Let $\varepsilon >0$, $ 0 < \rho_1 < \rho_2 < r $ and let $u \in BV(\Omega;\{\pm 1\}) $ be such that $u = u_{x,\nu }$ in a neighborhood of $\partial B_{\rho_1}(x)$ and there holds
 \begin{align*}
E(u,B_{\rho_1}(x)) \leq  m(x,\nu,\rho) +\varepsilon.
 \end{align*}
 Define $\tilde{u} \in BV(\Omega;\{\pm 1\})$ by
 \begin{align*}
 \tilde{u}(z) =\begin{cases} u(z) &\text{if } z \in B_{\rho_1}(x)\\
 u_{x,\nu}(z) &\text{otherwise.}
 \end{cases}
 \end{align*}
 Note that $\tilde{u} = u_{x,\nu}$ in a neighborhood of $\partial B_{\rho_2}(x)$ hence we obtain
 \begin{align*}
 m(\rho_2) &\leq E(\tilde{u},B_{\rho_2}(x)) +\int_{\Pi_\nu(x) \cap \left(B_r(x) \setminus B_{\rho_2}(x)\right)} g(y,\nu(y)) \mathrm{d}\mathcal{H}^{d-1} \\&\leq E(u,B_{\rho_1}(x)) + E(\tilde{u},B_{\rho_2}(x)\setminus B_{\rho_1}(x)) + \int_{\Pi_\nu(x) \cap \left(B_r(x) \setminus B_{\rho_2}(x)\right)} \varphi(y,\nu(y)) \mathrm{d}\mathcal{H}^{d-1} \\&\leq E(u,B_{\rho_1}(x)) + \int_{\Pi_\nu(x) \cap \left(B_r(x) \setminus B_{\rho_1}(x)\right)} g(y,\nu(y)) \mathrm{d}\mathcal{H}^{d-1}\\&\leq m(x,\nu,\rho_1) + \int_{\Pi_\nu(x) \cap \left(B_r(x) \setminus B_{\rho_1}(x)\right)} g(y,\nu(y)) \mathrm{d}\mathcal{H}^{d-1} +\varepsilon ,
 \end{align*}
 where we used in the last inequality the fact that $\mathcal{H}^{d-1}(\partial^*\{\tilde{u}=1\} \cap \partial B_{\rho_1}) = 0$. The claim follows letting $\varepsilon \to 0$.
 Hence $ \rho \mapsto m(\rho)$ has countably many discontinuity points $E = E(x,\nu)$. Moreover since $\mathcal{H}^{d-1}(\partial^*\{ u_{x,\nu}=1\}\cap \partial B_{\rho}) = 0$ for all $0 < \rho < r$ we have that
 \begin{align*}
  \rho \mapsto \int_{\Pi_\nu(x) \cap\left( B_r(x) \setminus B_{\rho}(x)\right)} g(y,\nu) \mathrm{d}\mathcal{H}^{d-1}
 \end{align*}
  is a continuous function. Hence we obtain
 \begin{align*}
 \rho \mapsto m(x, \nu ,\rho) = m(\rho) - \int_{\Pi_\nu(x) \cap \left( B_r(x) \setminus B_{\rho}(x)\right)} g(y,\nu) \mathrm{d}\mathcal{H}^{d-1} 
 \end{align*}
 is continuous for all but countably many $\rho \in (0,r)$. 
 
 \medskip

(ii) By [\cite{BFM},Lemma 3.1] it holds that
 \begin{align}
 |m(x,\nu_1,\rho)-m(x,\nu_2,\rho)| \leq C \int_{\partial B_\rho(x)}|\mathrm{tr}(u_{x,\nu_1}-u_{x,\nu_2})|\mathrm{d}\mathcal{H}^{d-1},
 \end{align}
 where 
 \begin{align*}
 \int_{\partial B_\rho(x)}|\mathrm{tr}(u_{x,\nu_1}-u_{x,\nu_2})|\mathrm{d}\mathcal{H}^{d-1} \leq C\rho^{d-1} \arccos(\langle\nu_1,\nu_2\rangle)\leq \rho^{d-1} w(|\nu_1-\nu_2|)
 \end{align*}
holds with $w$ a modulus of continuity and the claim follows.

\medskip

 (iii) \underline{Claim:} For all $0 < \rho_1 < \rho_2 <r$ , $x_1, x_2 \in \Omega$ such that $|x_1-x_2| < \min\{ \rho_1, |\rho_1 -\rho_2|\}$, there holds
 \begin{align} \label{Ineq}
 m(x_2,\nu,\rho_2) \leq m(x_1,\nu,\rho_1) + C \rho_1^{d-2}|x_1-x_2| +  C\Bigl(\rho_2^{d-1} -  \Bigl(\sqrt{\rho_1^2-|x_1-x_2|^2}\Bigr)^{d-1}\Bigr)
 \end{align} 
 \underline{Proof of the claim:} Let $\varepsilon >0$, $u \in BV(\Omega, \{\pm 1\})$ be such that $u = u_{x_1,\nu}$ in a neighborhood of $\partial B_{\rho_1}(x_1)$ and
 \begin{align*}
 E(u,B_{\rho_1}(x_1)) \leq m(x_1,\nu,\rho_1) +\varepsilon
 \end{align*}
 Let $\tilde{u} \in BV(\Omega,\{ \pm 1\})$ be defined by
 \begin{align*}
 \tilde{u}(z) = \begin{cases} u(z) &,\text{if } z \in B_{\rho_1}(x_1),\\
u_{x_2,\nu}(z) &,\text{otherwise}.
 \end{cases}
 \end{align*}
 We then have $\tilde{u} = u_{x_2,\nu} $ in a neighborhood of $\partial B_{\rho_2}(x_2)$ and 
 \begin{align*}
 m(x_2,\nu,\rho_2)&\leq E(\tilde{u},B_{\rho_2}(x_2)) \leq  E(\tilde{u},B_{\rho_1}(x_1)) +   E(\tilde{u},B_{\rho_2}(x_2)\setminus B_{\rho_1}(x_1)) \\ &\leq E(u,B_{\rho_1}(x_1)) + C \mathcal{H}^{d-1}(\partial^*\{\tilde{u}=1\} \cap \partial B_{\rho_1}(x_1))  \\& \qquad+C \mathcal{H}^{d-1}(\partial^*\{u_{x_2,\nu}=1\} \cap \left( B_{\rho_2}(x_2) \setminus \overline{B}_{\rho_1}(x_1)\right)) \\& \leq E(u,B_{\rho_1}(x_1)) + C \rho_1^{d-2} |x_1-x_2| + C\Bigl(\rho_2^{d-1} -  \Bigl(\sqrt{\rho_1^2-|x_1-x_2|^2}\Bigr)^{d-1}\Bigr)\\&\leq m(x_1,\nu,\rho_1) +C \rho_1^{d-2} |x_1-x_2| + C\Bigl(\rho_2^{d-1} -  \Bigl(\sqrt{\rho_1^2-|x_1-x_2|^2}\Bigr)^{d-1}\Bigr) +\varepsilon.
 \end{align*}
The claim follows by letting $\varepsilon \to 0$.\\
\underline{Claim:}  Let $x_0 \in \Omega, \rho_0 \in (0,\mathrm{dist(x_0,\partial \Omega)}), \nu \in S^{d-1}$ and assume that $ \rho \mapsto m(x_0,\rho,\nu)$ is continuous at $\rho_0 $, then $x \mapsto m(x,\rho_0,\nu) $ is continuous at $x_0$. \\
\underline{Proof of the claim:} Fix $\varepsilon >0$. First we prove that 
\begin{align*}
m(x,\nu,r_0) \geq m(x_0,\nu,r_0) - \varepsilon \text{ for all } |x-x_0| < \delta=\delta(\varepsilon,r_0)
\end{align*}
To this end let $r_\varepsilon > r_0,0<\delta < \min\{r_0,|r_\varepsilon-r_0|\}$ be such that
\begin{align} \label{2}
 |m(x_0,\nu,r) - m(x_0,\nu,r_0)| < \frac{\varepsilon}{2}, \quad r_\varepsilon^{d-1} - \Bigl(\sqrt{r_0^2-\delta^2}\Bigr)^{d-1} \leq \frac{\varepsilon}{4C},\quad \delta r_\varepsilon^{d-2} \leq \frac{\varepsilon}{4Cr_0}.
\end{align}
We then have by (\ref{Ineq}),(\ref{2})
\begin{align*}
m(x_0,\nu,r_0)& \leq  m(x_0,\nu,r_\varepsilon)+\frac{\varepsilon}{2}  \leq m(x,\nu,r_0) + Cr_0^{d-2}\delta +C\Bigl( r_\varepsilon^{d-1} - \Bigl(\sqrt{r_0^2-\delta^2}\Bigr)^{d-1}\Bigr)\\& \leq m(x,\nu,r_0) + \varepsilon 
\end{align*}
for all  $|x-x_0| <\delta$.
On the other hand by (\ref{Ineq}),(\ref{2}) we have
\begin{align*}
m(x,\nu,r_0) &\leq m(x_0,\nu,r_\varepsilon) + Cr_0^{d-2}\delta +C\Bigl( r_\varepsilon^{d-1} - \Bigl(\sqrt{r_0^2-\delta^2}\Bigr)^{d-1}\Bigr)\\&\leq m(x_0,\nu,r_0) +Cr_0^{d-2}\delta +C\Bigl( r_\varepsilon^{d-1} - \Bigl(\sqrt{r_0^2-\delta^2}\Bigr)^{d-1}\Bigr)+\frac{\varepsilon}{2} \\
&\leq m(x_0,\nu,r_0) +\varepsilon \text{ for all } |x-x_0| <\delta
\end{align*}
which yields the claim.
 \end{proof}
 
 \begin{remark}\label{equality} \rm Note that if there exist $m_1,m_2 : \Omega \times S^{d-1} \times (0,\mathrm{dist}(x,\partial \Omega)) \to [0,\infty)$ satisfying (i)-(iii) of Proposition \ref{Propertiesm} and there exists $\mathcal{D}_1 \times \mathcal{D}_2 \subset \Omega \times S^{d-1}$ countable and dense such that for all $x \in \mathcal{D}_1 $ there exists $\mathcal{D}_3(x) \subset (0,\mathrm{dist}(x,\partial \Omega))$ countable and dense and there holds
 \begin{align}\label{63}
 m_1(x,\nu,\rho)=m_2(x,\nu,\rho)
 \end{align}
 for all $(x,\nu,\rho) \in \mathcal{D}_1 \times \mathcal{D}_2 \times \mathcal{D}_3(x)$, then for all $(x,\nu) \in \Omega\times S^{d-1}$ we have that
 \begin{align*}
 m_1(x,\nu,\rho)=m_2(x,\nu,\rho)
\end{align*}  
for all $\rho \in (0,\mathrm{dist}(x,\partial \Omega)) \setminus E(x)$, where $E(x) $ is countable. 

We prove the claim. By (ii) it suffices to prove the equality (\ref{63})  on $\Omega \times \mathcal{D}_2 \times (0,\mathrm{dist}(x,\partial \Omega))$, with $E(x) $ countable. We set
\begin{align*}
E(x) = \Bigl(\bigcup_{\nu \in \mathcal{D}_2} E(x,\nu) \Bigr) \cup \Bigl(\bigcup_{y \in \mathcal{D}_1,\nu \in \mathcal{D}_2}E(y,\nu)\Bigr)
\end{align*}
 where $E(z,\nu)$ is the countable set of discontinuity points of $m_1(z,\nu,\cdot) $ and $m_2(z,\nu,\cdot) $ given by Proposition \ref{Propertiesm}. If $x \in \mathcal{D}_1,\nu \in S^{d-1}$, then we have that $ m_1(x,\nu,\rho)=m_2(x,\nu,\rho)$ for all $\rho $ in a countable dense set $\mathcal{D}_3$ and both $m_1(x,\nu,\cdot)$ and $m_2(x,\nu,\cdot)$ are continuous on $(0,\mathrm{dist}(x,\partial \Omega)) \setminus E(x) $, therefore for every $\rho \in (0,\mathrm{dist}(x,\partial \Omega)) \setminus E(x)$ we can find $\rho_k \to \rho$, $\{\rho_k\}_k \subset \mathcal{D}_3$ such that
 \begin{align*}
 m_1(x,\nu,\rho)= \lim_{k \to \infty} m_1(x,\nu,\rho_k) = \lim_{k \to \infty} m_2(x,\nu,\rho_k) = m_2(x,\nu,\rho)
 \end{align*}
 Let now $(x_0,\nu,\rho_0) \in \Omega \times S^{d-1} \times ((0,\mathrm{dist}(x_0,\partial \Omega))\setminus E(x_0)$. By the definition of $E(x_0)$ we have that $\rho \mapsto m_1(x_0,\nu,\rho),\rho \mapsto m_2(x_0,\nu,\rho)$ are both continuous at $\rho_0$, by (iii) of Proposition \ref{Propertiesm} there holds that $x \mapsto m_1(x,\nu,\rho_0),\rho \mapsto m_2(x,\nu,\rho_0)$ is continuous at $x_0$. Choose $x_k \to x_0$, $\{x_k\}_k \subset \mathcal{D}_1$, by the definition of $E(x_0)$ we have that
 \begin{align*}
 m_1(x_0,\nu,\rho_0)= \lim_{k \to \infty} m_1(x_k,\nu,\rho_0) = \lim_{k \to \infty} m_2(x_k,\nu,\rho_0) = m_2(x_0,\nu,\rho_0)
 \end{align*}
 and the remark holds true.
 \end{remark}
 The next goal is to prove Theorem \ref{convergence} below, which relates $\Gamma$-convergence with the convergence of the corresponding minimum problems (\ref{minimumproblem}) and we then use it in the proof of Proposition \ref{Density of Continuos integrands}. In order to prove Theorem \ref{convergence} we apply Lemma \ref{Functionalequality}, which shows, that every lower semicontinuous surface energy functional is characterized by its infimum problems on balls.
  \begin{lemma}\label{Functionalequality}
  Let $\varphi_i : \Omega\times \mathbb{R}^d \to [0,+\infty), i = 1,2$ be such that the trivial bounds (\ref{tri}) are satisfied and let $E_i : BV(\Omega,\{\pm 1\}) \times \mathcal{A}(\Omega) \to [0,+\infty) $ be defined by
  \begin{align*}
  E_i(u,A)= \int_{\partial^*\{u=1\}\cap A} \varphi_i(x,\nu_u(x)) \mathrm{d}\mathcal{H}^{d-1} 
  \end{align*}
  for all $(u,A) \in BV(\Omega,\{\pm 1\}) \times \mathcal{A}(\Omega)$.
For all $A \in \mathcal{A}(\Omega)$ let  $E_i(\cdot,A)$ be $L^1(A)$-lower semicontinuous and assume that
  \begin{align} \label{m1=m2}
  m_1(x,\rho,\nu) = m_2(x,\rho,\nu) \text{ for all } x \in \Omega, \nu \in S^{d-1},  \rho \in (0,\mathrm{dist}(x,\partial \Omega)) \setminus E(x) 
  \end{align}
  where $E(x) $ is a countable set, then
  \begin{align*}
  E_1(u,A) = E_2(u,A) 
  \end{align*}
for all $(u,A) \in BV(\Omega, \{\pm 1\} ) \times \mathcal{A}(\Omega)$.
  \end{lemma}
  \begin{proof}
  Let $u \in BV(\Omega,\{\pm 1\})$ and define 
  \begin{equation}   \label{Cover}
  \begin{split}
  \mathcal{Q}^\delta = \Big\{ &B_i^\delta : B_i^\delta = B_{\rho_i}(x_i), B_i^\delta \subset A , x_i \in \partial^*\{u=1\} , \rho_i < \delta, B_i^\delta \cap B_j^\delta=\emptyset, i \neq j \\& (\ref{m1=m2}) \text{ is satisfied}, \int_{\partial B_i^\delta} | \mathrm{tr}(u - u_{x_i,\nu_i} )|\mathrm{d}\mathcal{H}^{d-1} < \rho_i^{d-1} \delta,\\& \mathcal{H}^{d-1}(\partial^*\{u=1\}) \setminus \bigcup^\infty_{i=1}B_i^\delta) = 0, 
  \mathcal{H}^{d-1}(\partial^*\{u=1\}) \cap B_i^\delta) \geq \frac{1}{2}w_{d-1}\rho_i^{d-1}\Big\}.
  \end{split}
\end{equation}
  By Besicovitch Covering Theorem we know that there exists such a countable cover. Let  $m : BV(\Omega;\{\pm 1\}) \times \mathcal{A}(\Omega)\to [0,+\infty)$ be defined by
  \begin{align*}
  m_i(u,A) = \inf\Big\{ E_i(v,A) : v \in BV(\Omega,\{\pm 1\}),\, u = v \text{ on a neighborhood of } A\Big\}
  \end{align*}
  for $i=1,2$. Note that by [\cite{BFM},Lemma 3.1] 
  \begin{align} \label{trace}
  |m_i(u,A) -m_i(v,A)| \leq C\int_{\partial A}| \mathrm{tr}(u - v) |\mathrm{d}\mathcal{H}^{d-1}
  \end{align}
for $i=1,2$  holds. Therefore by (\ref{Cover}),(\ref{trace}) we have
  \begin{align*}
  E_1(u,A) &\geq \sum^{\infty}_{i=1}E_1(u,B^\delta_i) \geq \sum^{\infty}_{i=1} m_1(u,B_i^\delta) \\& \geq \sum^{\infty}_{i=1}m_1(u_{x_i,\nu_i},B^\delta_i) - C \sum^{\infty}_{i=1}\int_{\partial B_i^\delta} | \mathrm{tr}(u - u_{x_i,\nu_i}) |\mathrm{d}\mathcal{H}^{d-1} \\&  \geq \sum^{\infty}_{i=1}m_2(u_{x_i,\nu_i},B^\delta_i) - C \delta \sum^{\infty}_{i=1}\rho_i^{d-1} \\& \geq \sum^{\infty}_{i=1}m_2(u,B^\delta_i) - C \delta \sum^{\infty}_{i=1}\rho_i^{d-1} \\&\geq \sum^{\infty}_{i=1}m_2(u,B^\delta_i) - C \delta \mathcal{H}^{d-1}(\partial^*\{u=1\}) \cap A)
  \end{align*}
  Now choose $u^i_\delta \in BV(B^\delta_i, \{ \pm 1\})$ such that $u^\delta_i = u$ in a Neighborhood of $\partial B^\delta_i$ and 
  \begin{align*}
  E_2(u^i_\delta,B^\delta_i) \leq \frac{1}{2^i} \delta + m_2(u, B^\delta_i).
  \end{align*}
  If we set $N^\delta = \Omega \setminus \bigcup^\infty_{i=1}B_i^\delta $ and define
  \begin{align*}
  u^\delta(x) =\begin{cases} u^i_\delta & x \in B_i^\delta \\
  u(x) & x \in N^\delta.
  \end{cases}
  \end{align*}
  By the coercivity assumption on $\varphi_2$ we have that $u^\delta \in BV(\Omega, \{\pm 1\})$  and 
  \begin{align*}
  E_2(u^\delta,A) &\leq \sum^{\infty}_{i=1}E_2(u^\delta,B^\delta_i) + E_2(u^\delta,N^\delta) \leq \sum^{\infty}_{i=1}m_2(u^\delta_i,B^\delta_i) +\delta \\& \leq E_1(u,A) + C(u) \delta.
  \end{align*}
  We claim that $u^\delta \to u$ in $L^1(A)$. In fact
  \begin{align*}
  ||u^\delta-u||_{L^1(A)} = \sum^\infty_{i=1}  ||u^\delta-u||_{L^1(B^\delta_i)} &\leq C \delta \sum^\infty_{i=1}\mathcal{H}^{d-1}(\partial^*\{u=1\} \cap B^\delta_i) \\&= C\delta \mathcal{H}^{d-1}(\partial^*\{u=1\}) \to 0 \text{ as } \delta \to 0.
  \end{align*}
  Therefore by the lower semicontinuity of $E_2$ we obtain
  \begin{align*}
  E_2(u,A) \leq \liminf_{\delta \to 0} E_2(u^\delta,A) \leq E_1(u,A).
  \end{align*}
  By exchanging the roles of $E_1 $ and $E_2$ we obtain the statement.
  \end{proof}
Set
  \begin{align*}
& m_n(x,\nu,\rho) = \inf \{ E_{\varphi_n}(u, B_\rho(x) ) : u = u_{x,\nu} \text{ in a neighborhood of } \partial B_\rho(x)\} \\&
m(x,\nu,\rho) = \inf \{ E_\varphi(u, B_\rho(x) ) : u = u_{x,\nu} \text{ in a neighborhood of } \partial B_\rho(x)\}
 \end{align*}
  \begin{theorem} \label{convergence} Let $E_{\varphi_n}, E_\varphi : BV(\Omega,\{\pm 1\}) \times \mathcal{A}(\Omega) \to [0,+\infty)$ be defined by 
  \begin{align*}
  E_{\varphi_n}(u,A) = \int_{\partial^*\{u=1\} \cap A} \varphi_n(x,\nu_u(x))\mathcal{H}^{d-1} 
\end{align*}
and
\begin{align*}
E_\varphi(u,A) = \int_{\partial^*\{u=1\} \cap A} \varphi(x,\nu_u(x)) \mathcal{H}^{d-1} 
  \end{align*}
for all $(u,A) \in BV(\Omega,\{\pm 1\}) \times \mathcal{A}(\Omega)$, where $\varphi_n, \varphi : \Omega \times \mathbb{R}^d \to [0,+\infty]$ are such that  $\varphi_n(x,\cdot),\varphi(x,\cdot)$ satisfy the trivial bounds (\ref{tri}) for all $x\in \Omega$.
  Then the following are equivalent:
  \begin{itemize}
  \item[i)] $E_{\varphi_n}(\cdot, A)\, \Gamma$-converges to $E_\varphi(\cdot, A)$ with respect to the $L^1(A)$-topology
  \item[ii)]$m_n(x,\nu,\rho) \to m(x,\nu,\rho)$ for all $(x,\nu,\rho)\in \Omega \times \in S^{d-1}\times (0,\mathrm{dist}(x,\partial \Omega)) \setminus E(x)$ where $E(x) $ is countable.
  \end{itemize}
  \end{theorem}
  \begin{proof}
 We first show that (i) implies (ii).
 
 \underline{Step 1:} We show that
  \begin{align*}
  \limsup_{n \to \infty} m_n(x,\nu,\rho) \leq m(x,\nu,\rho) \text{ for all } (x,\nu,\rho) \in \Omega \times S^{d-1} \times (0,\mathrm{dist}(x,\partial \Omega))
  \end{align*}
  To this end let $\varepsilon >0$ and $u \in BV(\Omega,\{\pm 1\}) $ be such that $u= u_{x,\nu} $ in a neighborhood of $\partial B_\rho(x)$ and 
  \begin{align*}
  E_\varphi(u,B_\rho(x)) \leq m(x,\nu,\rho) + \varepsilon
  \end{align*}
   Since $E_n \overset{\Gamma}{\to} E$ there exists  $u_n \in BV(\Omega,\{\pm 1\})$ such that
  \begin{align*}
  \limsup_{n \to \infty} E_{\varphi_n}(u_n,B_\rho(x)) \leq E_\varphi(u,B_\rho(x)).
  \end{align*}
  By a cut-off argument we construct $\tilde{u}_n \in BV(\Omega,\{ \pm 1\})$ such that $\tilde{u}_n = u_{x,\nu}$ in a neighborhood of $\partial B_\rho(x)$ and 
  \begin{align*}
  \limsup_{n \to \infty} E_{\varphi_n}(\tilde{u}_n,B_\rho(x)) \leq \limsup_{n \to \infty} E_{\varphi_n}(u_n,B_\rho(x)).
  \end{align*}
  Hence we have
  \begin{align*}
 \limsup_{n \to \infty} m_n(x,\nu,\rho)&\leq  \limsup_{n \to \infty} E_{\varphi_n}(\tilde{u}_n,B_\rho(x))\leq \limsup_{n \to \infty} E_{\varphi_n}(u_n,B_\rho(x)) \\&\leq E_\varphi(u,B_\rho(x))\leq m(x,\nu,\rho) + \varepsilon.
  \end{align*}
  The claim follows as $\varepsilon \to 0$. By proposition (\ref{Propertiesm}) (ii) we have that that 
  \begin{align}\label{limsup}
  \limsup_{\rho' \to \rho} m(x,\nu,\rho') = m(x,\nu,\rho)
  \end{align}
  for all but countably many $\rho \in (0,\mathrm{dist}(x,\partial \Omega))$, where $\rho'$ is converging decreasingly to $\rho$. 
  
  \underline{Step 2:} We show that 
  \begin{align*}
  \liminf_{n \to \infty} m_n(x,\nu,\rho) \geq \limsup_{\rho' \to \rho} m(x,\nu,\rho')
  \end{align*}
  with $\rho'$ converging decreasingly to $\rho$.
  To prove this choose for all $n \in \mathbb{N}$, $u_n \in BV(\Omega,\{\pm 1\})$ such that $u_n = u_{x,\nu}$ in a neighborhood of $\partial B_\rho(x)$ and 
  \begin{align*}
  E_{\varphi_n}(u_n,B_\rho(x)) \leq m_n(x,\nu,\rho) + \frac{1}{n}.
  \end{align*}
  Let $\rho' > \rho $ and define
  \begin{align*}
  \tilde{u}_n(z) = \begin{cases} u_n(z) & z \in B_\rho(x) \\
  u_{x,\nu}(z) & \text{otherwise.}
  \end{cases}
  \end{align*}
  We have
  \begin{align*}
  m_n(x,\nu,\rho)\geq E_{\varphi_n}(u_n,B_\rho(x)) \geq E_{\varphi_n}(\tilde{u},B_{\rho'}(x)) -\frac{1}{n}-C|\rho-\rho'|.
  \end{align*}
  By the coercivity assumption we know that up to subsequences $\tilde{u}_n \to v$ and therefore  that $v= u_{x,\nu} $ in a neighborhood of $\partial B_{\rho'}(x)$ we obtain
  \begin{align*}
  \liminf_{n \to \infty} m_n(x,\nu,\rho) \geq \liminf_{ n \to \infty} E_{\varphi_n}(u_n,B_\rho(x)) \geq E_{\varphi}(v, B_{\rho'}(x)) - C|\rho-\rho'|
  \end{align*}
  By (\ref{limsup}) the claim follows for all such $\rho \in (0,1) \setminus E(x)$ as $\rho'$ converges decreasingly to $\rho$. 
  
  \medskip
  
  Now we prove that (ii) implies (i): Take a subsequence (not relabeled) $\{E_{\varphi_n}\}_n$. By [\cite{AB},Theorem 3.2] up to subsequences $E_{\varphi_n}(\cdot,A) \, \Gamma$-converges to some $\tilde{E} : BV(\Omega,\{\pm 1\}) \times \mathcal{A}(\Omega) \to [0,+\infty)$ of the form 
  \begin{align*}
  \tilde{E}(u,A) = \int_{\partial^*\{u=1\} \cap A} \tilde{\varphi}(x,\nu_u(x)) \mathrm{d}\mathcal{H}^{d-1} \text{ for all } (u,A) \in BV(\Omega,\{\pm 1\}) \times A(\Omega)
\end{align*}   
and therefore if we denote by $\tilde{m}$ the associated minimum problems of the energy $\tilde{E}$ by (i) implies (ii) and by our assumption we have that
\begin{align*}
\tilde{m}(x,\rho,\nu) =\lim_{n \to \infty} m_n(x,\rho,\nu) = m(x,\rho,\nu) 
\end{align*}
 for all  $x,\nu,\rho \in \Omega \times S^{d-1} \times  (0,\mathrm{dist}(x,\partial \Omega)) \setminus E(x)$
where $E(x)$ is a countable set. By Lemma \ref{Functionalequality} we have that $\tilde{E} = E_\varphi$. Therefore every subsequence contains a further subsequences which $\Gamma$-converges to $E_\varphi$. By the Urysohn-property of $\Gamma$-convergence we have that 
$E_{\varphi_n}$ $\Gamma$-converges to $E_\varphi$.
\end{proof} 
Since the energies that are involved are equi-coercive we may use a metrizability argument for $\Gamma$-convergence [\cite{DM},Theorem 10.22]. We therefore may argue by a diagonalization procedure. To this end we need to define a {\em minimal volume fraction} $\theta_\varphi$  and establish some properties of it. Those are contained in the next lemma.

 For $\varphi :  \mathbb{R}^d \to [0,+\infty)$, even, convex, positively $1$-homogeneous, satisfying (\ref{tri}), or equivalently $\varphi \in \textbf{H}_{\alpha,\beta,V}(1)$, we define $\theta_\varphi \in [0,1]$ by
  \begin{align} \label{minthetavarphi}
  \theta_\varphi = \min \{ s \in [0,1] : \varphi \in \textbf{H}_{\alpha,\beta,V}(s)\} 
  \end{align}
  and we denote
  \begin{align*}
  B_r(\varphi) = \{ \varphi' : \text{ even, convex, positively $1$-homogenous and } \mathrm{d}(\varphi,\varphi') < r \} ,
\end{align*}
where
\begin{align*}  
  \mathrm{d}(\varphi,\varphi') = \sup_{\nu \in S^{d-1}} |\varphi(\nu) - \varphi'(\nu)|.
  \end{align*}
  Note that the minimum in (\ref{minthetavarphi}) is attained by the definition of $\textbf{H}_{\alpha,\beta,V}(\theta)$.
  If $\varphi : \Omega \times \mathbb{R}^d \to [0,+\infty)$ is such that $\varphi(x,\cdot) \in \textbf{H}_{\alpha,\beta,V}(1)$ for all $x \in \Omega$,  we define $\theta_\varphi : \Omega \to [0,1]$ by 
  \begin{align*}
  \theta_\varphi(x) =  \theta_{\varphi(x,\cdot)}.
  \end{align*}

  \begin{lemma} \label{properties of theta}  The following properties hold true

\begin{itemize}  

 \item[i)] Let $0 \leq \theta_1 < \theta_2 \leq 1 $ and $\varphi \in \textbf{H}_{\alpha,\beta,V}(\theta_1)$ then there exists $r>0$ such that $B_r(\varphi) \cap \textbf{H}_{\alpha,\beta,V}(1) \subset \textbf{H}_{\alpha,\beta,V}(\theta_2)$.
 \item[ii)] Let $\varphi,\varphi_n \in \textbf{H}_{\alpha,\beta,V}(1)$ and $\mathrm{d}(\varphi_n,\varphi) \to 0$ as $n \to \infty$, then  $\theta_{\varphi_n} \to \theta_\varphi$.
 \item[iii)] Let  $\varphi(\cdot,\nu)$ be continuous for all $\nu \in S^{d-1}$, then $\theta_\varphi$ is continuous.
 \item[iv)] For every surface energy density $\varphi$  the function $\theta_\varphi$ is measurable.
\end{itemize}
  \end{lemma}
  \begin{proof}

  (i) Let $0 \leq \theta_1 < \theta_2 \leq 1$ and $\varphi \in \textbf{H}_{\alpha,\beta,V}(\theta_1)$. Then for all $\varphi' \in B_r(\varphi) \cap \textbf{H}_{\alpha,\beta,V}(1)$ we have
  \begin{align*}
  \varphi'(\nu) \leq \varphi(\nu) + r \leq \sum_{\xi \in V}c_\xi |\langle \nu,\xi\rangle| +  r \leq \sum_{\xi \in V}(c_\xi+r) |\langle \nu,\xi\rangle|=: \sum_{\xi \in V}\tilde{c}_\xi |\langle \nu,\xi\rangle|,
  \end{align*}
  where $ \alpha_\xi \leq c_\xi \leq (\theta_\xi\beta_\xi+(1-\theta_\xi)\alpha_\xi) $ for some  $\{\theta_\xi\}_{\xi \in V} \subset [0,1]$ satisfying (\ref{Averagefraction}) with $\theta_1$ and
  \begin{align*}  
  \tilde{c}_\xi &\leq (\theta_\xi\beta_\xi+(1-\theta_\xi)\alpha_\xi)+r= \Bigl(\Bigl(\theta_\xi+\frac{r}{\beta_\xi-\alpha_\xi}\Bigr)\beta_\xi+\Bigl(1-(\theta_\xi+\frac{r}{\beta_\xi-\alpha_\xi})\Bigr)\alpha_\xi\Bigr) \\&\leq (\tilde{\theta}_\xi\beta_\xi+(1-\tilde{\theta}_\xi)\alpha_\xi), 
  \end{align*}
if $\tilde{\theta}_\xi -\theta_\xi \geq \frac{r}{\beta_\xi-\alpha_\xi}$ and $\{\tilde{\theta}_\xi\}_\xi \subset [0,1]$ satisfy (\ref{Averagefraction}) with $\theta_2$. Set $$ r :=\min_{\xi \in V} \frac{(\tilde{\theta}_\xi-\theta_\xi)(\beta_\xi-\alpha_\xi)}{2}.$$ We have that $B_r(\varphi) \cap \textbf{H}_{\alpha,\beta,V}(1) \subset \textbf{H}_{\alpha,\beta,V}(\theta_2)$. 

\medskip

  (ii) Let $\mathrm{d}(\varphi_n,\varphi) \to 0$ as $n \to \infty$. Up to subsequences we have that $\theta_{\varphi_n} \to \tilde{\theta}$. By the definition of $\theta_\varphi$ we have that $\theta_\varphi \leq \tilde{\theta}$, since $\varphi \in \textbf{H}_{\alpha,\beta,V}(\tilde{\theta})$. Assume that $\theta_\varphi < \theta < \tilde{\theta}$ for some $\theta \in (0,1)$. By (i) there exists a neighborhood $B_r(\varphi)$ of $\varphi $ such that $B_r(\varphi) \cap \textbf{H}_{\alpha,\beta,V}(1) \subset \textbf{H}_{\alpha,\beta,V}(\theta)$. Therefore for $n$ large enough we have that $\varphi_n \in \textbf{H}_{\alpha,\beta,V}(\theta)$ so that $ \theta_{\varphi_n} \leq \theta$, which contradicts $\theta_{\varphi_n} \to \tilde{\theta}$. 

\medskip  
  
  (iii) is a direct consequences of (ii). As for (iv) it suffices to notice that if we define $\varphi_n = \rho_n * \varphi$, where $\rho_n $ is a sequence of convolution kernels, we have that $\varphi_n$ is a sequence of continuous functions and $\varphi_n $ converges a.e. to $\varphi$. In view of (ii),(iii)  $\theta_\varphi$ is a a.e. limit of a sequence of continuous functions, hence it is measurable.
  \end{proof}

 \begin{proposition} \label{Density of Continuos integrands} Let $\theta : \Omega \to [0,1]$ be measurable and $\varphi : \Omega \times \mathbb{R}^d \to [0,+\infty]$ be such that $\varphi(x,\cdot) \in \textbf{H}_{\alpha,\beta,V}(\theta(x))$ for a.e. $x \in \Omega$. Then there exist sequences $\{\theta_n\}_n, \{c^n_\xi\}_{\xi,n}  \subset C(\Omega),0 \leq \theta_n \leq 1$ such that $\varphi_n : \Omega \times \mathbb{R}^d \to [0,+\infty]$ defined by 
 \begin{align}\label{densityform}
 \varphi_n(x,\nu)= \sum_{\xi \in V} c_\xi^n(x)|\langle \nu,\xi\rangle| 
 \end{align}
 satisfy $\varphi_n(x,\cdot) \in \textbf{H}_{\alpha,\beta,V}(\theta_n(x))$ for all $x \in \Omega$,  $E_{\varphi_n}$  $\Gamma$-converges to $E_\varphi$ and  $\theta_n \overset{*}{\rightharpoonup} \theta$.
  \end{proposition}
  \begin{proof} Since the energies involved are all equicoercive by the metrizability-properties of $\Gamma$-convergence (see [\cite{DM},Theorem 10.22])  we can use  a diagonal argument. It suffices to construct lower semicontinuous densities of the form (\ref{densityform}) such that the associated energies $\Gamma$-converge to $E_\varphi$. Every such function can be approximated from below by functions of the  form (\ref{densityform}) with continuous coefficients. Therefore the associated energies $\Gamma$-converge to the energy associated to the limit density and by  Lemma \ref{properties of theta} (ii) the associated local volume fractions converge weakly*.

 We now prove that there exists $\{\varphi_n\}_n$ of the form (\ref{densityform})
with $c_\xi^n: \Omega \to [0,\infty)$ lower semicontinuous and $E_{\varphi_n}$ $\Gamma$-converges to $E_\varphi$. By Theorem \ref{convergence} and Remark \ref{equality} it suffices to find $\varphi_n$ such that
  \begin{align*}
\lim_{n \to \infty}  m_n(x_i,\nu_i,\rho_i) = m(x_i,\nu_i,\rho_i) 
  \end{align*}
  where $\{x_i,\rho_i,\nu_i\}_{i \in \mathbb{N}} \subset \Omega \times (0,\mathrm{dist}(x,\partial \Omega)) \times S^{d-1}$ is a dense set. Moreover we assume,  by Besicovitch Covering Theorem and Remark \ref{equality},  that $|\Omega \setminus \bigcup_{i=1}^\infty  B_{\rho_i}(x_i)|=0 $. Let $u^n_i \in BV(\Omega,\{\pm 1\})$ be such that $u_i^n = u_{x_i,\nu_i} $ in a neighborhood of $\partial B_{\rho_i}(x_i)$ and
  \begin{align*}
  E_\varphi(u^n_i,B_{\rho_i}(x_i)) \leq m(x_i,\rho_i,\nu_i) +\frac{1}{n}.
  \end{align*}
  By Lusin's theorem and the rectifiability of $\partial^*\{u^n_i=1\}$ there exists $K_{i,n},C_{i,n} \subset \mathbb{R}^d $ compact and $a_\xi : K_{i,n} \to [0,\infty) $ such that
  \begin{itemize}
  \item[i)] $K_{i,n} \subset \partial^*\{u^n_i=1\} \cap B_{\rho_i}(x_i)$ , $x \mapsto\nu_{u^n_i}(x) $, $x \mapsto \varphi(x,\nu_{u^n_i}(x)) $ are continuous on $K_{i,n}$ and
    \begin{align*}
  \mathcal{H}^{d-1}(\partial^*\{u^n_i=1\} \setminus K_{i,n}) < \frac{1}{n}.
  \end{align*}
  \item[ii)] $C_{i,n} \subset B_{\rho_i}(x_i) \setminus K_{i,n}$ such that  
  \begin{align*}
  |B_{\rho_i}  \setminus C_{i,n}|=|(B_{\rho_i} \setminus K_{i,n}) \setminus C_{i,n}| < \frac{1}{n}2^{-i}.
  \end{align*}
and 
  \begin{align*}
  \varphi(x,\nu) \leq \sum_{\xi \in V} a_\xi(x)|\langle\nu,\xi\rangle|
  \end{align*}
for all $ x \in C_{i,n}, \nu \in S^{d-1}$. Moreover for all $\xi_0 \in V$ there holds
  \begin{align*}
  \varphi\Big(x,\frac{\xi_0}{||\xi_0||}\Big)= \sum_{\xi \in V} a_\xi(x)\Big|\langle \frac{\xi_0}{||\xi_0||},\xi\rangle\Big|.
  \end{align*}
    \end{itemize}
Define
  \begin{align*}
  \varphi_{n,i}(x,\nu) = \begin{cases} \sum_{\xi \in V} c_\xi(x,\nu_{u^n_i}(x)) |\langle\nu,\xi\rangle| & x \in K_{i,n} \\
  \sum_{\xi \in V} a_\xi(x) |\langle \nu,\xi\rangle| & x \in C_{i,n} \\
\sum_{\xi \in V}  \beta_\xi |\langle \nu,\xi\rangle| &\text{otherwise,}
  \end{cases}
  \end{align*}
  where $ c_\xi(x,\nu_{u^n_i}(x)) \in C(\Omega)$ are chosen  such that 
  \begin{align*}
  \sum_{\xi \in V}c_\xi(x,\nu_{u^n_i}(x))|\langle\nu_{u^n_i}(x),\xi\rangle|= \varphi(x, \nu_{u^n_i}(x))
  \end{align*}
  and set 
\begin{align*}
  \varphi_n(x,\xi) := \min_{1 \leq i \leq n} \varphi_{n,i}(x,\xi).
\end{align*}  
$\varphi_n(\cdot,\xi)$ is lower semicontinuous and for all $ 1 \leq i \leq n$ we have 
\begin{align*}
m_n(x_i,\nu_i,\rho_i) &\leq E_{\varphi_n}(u^n_i,B_{\rho_i}(x_i)) \\& \leq \int_{K_{i,n}} \varphi_n(x,\nu_{u_i^n}(x))\mathrm{d}\mathcal{H}^{d-1} + \int_{\partial^*\{u=1\} \setminus K_{i,n}} \varphi_n(x,\nu_{u_i^n}(x))\mathrm{d}\mathcal{H}^{d-1} \\& \leq \int_{\partial^*\{u=1\}} \varphi(x,\nu_{u_i^n}(x))\mathrm{d}\mathcal{H}^{d-1} + C \mathcal{H}^{d-1}(\partial^*\{u^n_i=1\} \setminus K_{i,n}) \\&\leq m(x_i,\nu_i,\rho_i) + \frac{C}{n}.
\end{align*}
Now let $\varepsilon >0$ and let $u \in BV(\Omega,\{ \pm 1\})$ be such that $u= u_{x_i,\nu_i}$ in a neighborhood of $\partial B_{\rho_i}(x_i)$  and
\begin{align*}
E_{\varphi_n}(u,B_{\rho_i}(x_i)) \leq m_n(x_i,\nu_i,\rho_i) + 
\varepsilon.
\end{align*}
We then have, noting that $\nu_u = \nu_{u^n_i} \, \mathcal{H}^{d-1}$-a.e. on $\partial^*\{u=1\}$, 
\begin{align*}
&m_n(x_i,\rho_i,\nu_i) + 
\varepsilon \geq E_{\varphi_n}(u,B_{\rho_i}(x_i)) = \int_{\partial^*\{u=1\} \cap B_{\rho_i}(x_i)} \varphi_n(x,\nu_u(x))\mathrm{d}\mathcal{H}^{d-1}\\ = &\sum^n_{j=1} \int_{I^n_j \cap B_{\rho_i}(x_i)} \varphi_n(x,\nu_u(x))\mathrm{d}\mathcal{H}^{d-1} + \int_{\left(\partial^*\{u=1\}\setminus \bigcup^n_{j=1} I^n_j\right) \cap B_{\rho_i}(x_i)} \varphi_n(x,\nu_u(x))\mathrm{d}\mathcal{H}^{d-1} \\ \geq &\sum^n_{j=1} \int_{I^n_j \cap B_{\rho_i}(x_i)} \varphi(x,\nu_u(x))\mathrm{d}\mathcal{H}^{d-1} + \int_{\left(\partial^*\{u=1\}\setminus \bigcup^n_{j=1} I^n_j\right) \cap B_{\rho_i}(x_i)} \varphi(x,\nu_u(x))\mathrm{d}\mathcal{H}^{d-1}\\=& E_\varphi(u,B_{\rho_i}(x_i)) \geq m(x_i,\nu_i,\rho_i),
\end{align*}
where $I^n_j = K_{j,n} \setminus \bigcup_{j' < j} K_{j',n}$. Letting $\varepsilon \to 0$ we obtain
\begin{align*}
|m_n(x_i,\nu_i,\rho_i)-m(x_i,\nu_i,\rho_i)|\leq \frac{C}{n} \text{ for all } 1 \leq i \leq n.
\end{align*}
Hence
\begin{align*}
m_n(x_i,\nu_i,\rho_i) \to m(x_i,\nu_i,\rho_i) \text{ for all } i \in \mathbb{N}.
\end{align*}
Therefore by Theorem \ref{convergence} we have that $E_{\varphi_n} $ $\Gamma$-converges to $E_\varphi$. 

It remains to show that $\theta_n \overset{*}{\rightharpoonup} \theta$.
By the definition of $\varphi_n$ we have that
 \begin{align*}
 \theta_{\varphi_n}(x) =\begin{cases} \theta(x) & x \in \bigcup_{i=1}^nC_{i,n}\\
 1 & x \notin \bigcup_{i=1}^n C_{i,n}
 \end{cases}
 \end{align*}
 where we remark that $|K_{i,n}|=0$. Now let $f \in L^1(\Omega)$ and  $\delta > 0$ let  $n_\delta \in \mathbb{N} $ be such that $|\Omega \setminus \bigcup^{n_\delta}_{i=1}B_{\rho_i}(x_i)| < \delta $. We have for $n $ big enough
 \begin{align*}
 \int_{\Omega} f (\theta_n-\theta)\mathrm{d}x = \int_{\Omega \setminus \bigcup^n_{i=1}C_{i,n}} f(\theta_n-\theta) \mathrm{d}x + \int_{ \bigcup^n_{i=1}C_{i,n}} f(\theta_n -\theta)\mathrm{d}x = \int_{\Omega \setminus \bigcup^n_{i=1}C_{i,n}} f(1-\theta) \mathrm{d}x.
 \end{align*}
 Note that $|f(1-\theta)| \leq 2f$ and 
\begin{align*}
 |\Omega \setminus \bigcup^{n_\delta}_{i=1}C_{i,n}|\leq |\Omega \setminus \bigcup^{n_\delta}_{i=1}B_{\rho_i}(x_i)| + | \bigcup^{n_\delta}_{i=1}B_{\rho_i}(x_i)\setminus C_{i,n}| \leq \delta + \frac{1}{n}.
\end{align*} 
Hence by the dominated convergence theorem we have that
\begin{align*}
|\int_{\Omega \setminus \bigcup^n_{i=1}C_{i,n}} f(1-\theta) \mathrm{d}x| \to 0
\end{align*}
as  $n \to \infty, \delta \to 0$.
Therefore we have that $\theta_n \overset{*}{\rightharpoonup} \theta$ as $n \to \infty$.
 \end{proof}
 \begin{proof}[Proof of Theorem \ref{Localization 1}:] By the metrizability-properties of $\Gamma$-convergence and of the weak*-convergence on bounded sets we proceed by successive approximation and conclude then by a diagonal argument.
 
 \medskip
 
\underline{Step 1:}  By Proposition \ref{Density of Continuos integrands} we can assume that 
\begin{align}\label{Step 1}
\varphi(x,\nu) = \sum_{\xi \in V}c_\xi(x)|\langle \nu,\xi\rangle| 
\end{align}
 for all  $(x,\nu) \in \Omega \times \mathbb{R}^d$
with $c_\xi \in C(\Omega,[\alpha_\xi,\beta_\xi])$, $\xi \in V$. Note that by (iii) of Lemma \ref{properties of theta} we have that $\theta_\varphi$ is continuous. 

\medskip

\underline{Step 2:} For every $\varphi$ of the form (\ref{Step 1}) let $c_\xi^k : \Omega \to [\alpha_\xi,\beta_\xi] $ be defined by
\begin{align*}
c^k_\xi(x)= \begin{cases}
\inf_{z \in Q_{2^{-k}}(x_n)} c_\xi(z) & x \in Q_{2^{-k}}(x_n), x_n \in \mathcal{Z}_k \\
\inf_{z \in Q_{2^{-k}}(x_n)\cup Q_{2^{-k}}(x_{n'})} c_\xi(z) & x \in  \overline{Q}_{2^{-k}}(x_n) \cap \overline{Q}_{2^{-k}}(x_{n'}), x_{n},x_{n'} \in \mathcal{Z}_k
\end{cases}
\end{align*}
with $\mathcal{Z}_k : = 2^{-k}\mathbb{Z}^d +\sum_{i=1}^d e_i 2^{-k-1}$.
If we define $\varphi_k : \Omega \times \mathbb{R}^d \to [0,+\infty)$ by
\begin{align} \label {Step 2}
\varphi_k(x,\nu) = \sum_{\xi \in V} c_\xi^k(x) |\langle \nu,\xi\rangle|
\end{align}
and $E_k : BV(\Omega,\{\pm 1\}) \to [0,+\infty)$
\begin{align*}
E_k(u) = \int_{\partial^*\{u=1\}\cap \Omega}\varphi_k(x,\nu_u(x))\mathrm{d}\mathcal{H}^{d-1}
\end{align*}
we have that 
\begin{align*}
\Gamma\text{-}\lim_{k \to \infty} E_k(u) = E(u).
\end{align*}
This follows since $c^k_\xi$ converges increasingly to $c_\xi$. Moreover by Lemma \ref{properties of theta} we have that $\theta_{\varphi_k} \overset{*}{\rightharpoonup}\theta_{\varphi}$ since $\varphi_k(x,\cdot) \to \varphi(x,\cdot) $ for all $x \in \Omega$.

\medskip

\underline{Step 3:} Every $E_k : BV(\Omega,\{\pm 1\}) \to [0,+\infty)$ can be approximated by $E^\delta_k : BV(\Omega,\{\pm 1\}) \to [0,+\infty)$ of the form 
\begin{align} \label{Step 3}
E_k^\delta(u) = \int_{\partial^*\{u=1\}\cap \Omega} \varphi^\delta_k(x,\nu_u(x))\mathrm{d}\mathcal{H}^{d-1}
\end{align}
where  $\varphi^\delta_k : \Omega \times \mathbb{R}^d \to [0,+\infty]$ is defined by
\begin{align*}
\varphi^\delta_k(x,\nu) = \sum_{\xi \in V} c_\xi^{k,\delta}(x)|\langle \nu,\xi\rangle|
\end{align*}
with
\begin{align*}
c_\xi^{k,\delta}(x) =\begin{cases} c_\xi^k(x) & x \in Q_{2^{-k}(1-\delta)}(x_n), x_n \in \mathcal{Z}_k \\
\beta_\xi &\text{otherwise.}
\end{cases}
\end{align*}

To prove Step 3 note first that since $c^{k,\delta}_\xi \geq c^k_\xi$ for all $\delta >0$, we have that 
\begin{align*}
\Gamma\text{-}\liminf_{\delta \to 0 } E^\delta_k(u) \geq E_k(u)
\end{align*}
By [\cite{AB},Theorem 3.2] we have that
\begin{align*}
\Gamma\text{-}\limsup_{\delta \to 0} E^\delta_k(u) = \int_{\partial^*\{u=1\} \cap \Omega} \tilde{\varphi}_k(x,\nu_u(x))\mathrm{d}\mathcal{H}^{d-1}
\end{align*}
for some $\tilde{\varphi} :\Omega \times \mathbb{R}^d \to [0,+\infty)$, where
\begin{align*}
\tilde{\varphi}_k(x,\nu) = \limsup_{\rho \to 0} \frac{m_k(x,\nu,\rho)}{w_{d-1}\rho^{d-1}}.
\end{align*}
Using [\cite{BFM},Lemma 4.3.5] we have that
\begin{align*}
\tilde{\varphi}_k(x,\nu) = \limsup_{\rho \to 0} \frac{m_k(x,\nu,\rho) }{w_{d-1}\rho^{d-1}}= \limsup_{\rho \to 0}\liminf_{\delta \to 0} \frac{m^\delta_k(x,\nu,\rho)}{w_{d-1}\rho^{d-1}} .
\end{align*}
Hence it follows that $\tilde{\varphi}(x,\nu) = \varphi(x,\nu)$  for all $x \in Q_{2^{-k}}(x_n), x_n \in \mathcal{Z}_k$. From this fact and by the lower semicontinuity of the energy functional we conclude,  as in the proof of (\ref{boundaryineq}), that $\tilde{\varphi}(x,\nu) \leq  \varphi(x,\nu)$ for all $x \in \partial Q_{2^{-k}}(x_n),\, x_n \in \mathcal{Z}_k$ with $\nu$ being the normal vector of $\partial Q_{2^{-k}}(x_k)$.  For every $\delta$  small enough we have that $c_\xi^{k,\delta} =c_\xi^k$ for all $x \in Q_{2^{-k}(1-\delta)}(x_n), x_n \in \mathcal{Z}_k$ with $\lim_{\delta \to 0}|\Omega \setminus \bigcup_{x_n \in \mathcal{Z}_k} Q_{2^{-k}(1-\delta)}(x_n)|=0$, therefore it follows that $\theta_{\varphi_k^\delta} \overset{*}{\rightharpoonup} \theta_{\varphi_k}$. 

\underline{Step 4:}  For $E^\delta_k : BV(\Omega,\{\pm 1\}) \to [0,+\infty)$  of the form  (\ref{Step 3}) we can find energies $E^{k,\delta}_{\varepsilon} : BV(\Omega,\{\pm 1\}) \to [0,+\infty)$ of the form
\begin{align} \label{Step 4}
E_\varepsilon^{k,\delta}(u) = \frac{1}{4}\sum_{\xi \in V}\sum_{i,i+\xi \in \Omega_\varepsilon} \varepsilon^{d-1} c^{k,\delta,\varepsilon}_{i,\xi} (u_{\varepsilon i}-u_{\varepsilon(i+\xi)})^2
\end{align}
for some $c^{k,\delta,\varepsilon}_{i,\xi} \in \{\alpha_\xi,\beta_\xi\}^{\mathbb{Z}^d}$ such that $\theta(\{c_{i,\xi}^{k,\delta,\varepsilon}\}) \overset{*}{\rightharpoonup} \theta_{\varphi^\delta_k} $.

Since $\varphi^\delta_k(x,\cdot) \in \textbf{H}_{\alpha,\beta,V}(\theta_{\varphi_{k}^\delta}(x_n))$ if $x \in Q_{2^{-k}(1-\delta)}(x_n), x_n \in \mathcal{Z}_k $ there exists a sequence of homogenized densities $\{\varphi^{k,\delta,n}_N\} $ such that $\varphi^{k,n}_N \to \varphi(x_n,\cdot)$ as $N \to \infty$ and for every $N \in \mathbb{N}$ there exist $\{c^{k,n,N}_{i,j}\}$ $T$-periodic for some $T \in \mathbb{N}$ such that $\theta(\{c^{k,n,N}_{i,\xi}\}) \to \theta (x_n) $,  where $\theta(\{c^{k,n,N}_{i,\xi}\})$ is  in the notation of (\ref{theta}), and
\begin{align*}
\Gamma\text{-}\lim_{\varepsilon \to 0} E^{k,n,N}_\varepsilon(u) = \int_{\partial^*\{u=1\} \cap \Omega}\varphi^{k,n}_N (\nu_u(x))\mathrm{d}\mathcal{H}^{d-1}.
\end{align*}  
We define $E_{k,N,\delta} : BV(\Omega,\{\pm 1\}) \to [0,+\infty)$ by
\begin{align} \label{last}
E_{k,N,\delta}(u) = \int_{\partial^*\{u=1\} \cap \Omega}\varphi^{k,\delta}_N (x,\nu_u(x))\mathrm{d}\mathcal{H}^{d-1},
\end{align}
with 
\begin{align*}
\varphi^{k,\delta}_N(x,\nu) =\begin{cases}
\varphi^{k,n}_N(\nu) & x  \in Q_{2^{-k}(1-\delta)}(x_n), x_n \in \mathcal{Z}_k \\
\sum_{\xi \in V} \beta_\xi |\langle \nu,\xi\rangle | &\text{otherwise.}
\end{cases}
\end{align*}
We have that
\begin{align*}
\Gamma\text{-}\lim_{N \to \infty} E_{k,N,\delta}(u)  = E_k^\delta(u).
\end{align*}
This follows from $ \varphi^{k,n}_N \to \varphi(x_n,\cdot)$ and therefore
\begin{align*}
\tilde{\varphi}^\delta_k(x,\nu) = \limsup_{\rho \to 0} \frac{m^\delta_k(x,\nu,\rho)}{w_{d-1}\rho^{d-1}} = \limsup_{\rho \to 0}\liminf_{N \to \infty} \frac{m^\delta_{k,N}(x,\nu,\rho)}{w_{d-1}\rho^{d-1}} .
\end{align*}  
Note that if we define $E^{k,N,\delta}_\varepsilon :BV(\Omega,\{\pm 1\}) \to [0,+\infty)$ by (\ref{Step 4}) with
\begin{align*}
c^{k,\delta,N,\varepsilon}_{i,\xi}= \begin{cases} c^{k,n,N}_{i,\xi} & i \in ( Q_{2^{-k}(1-\delta)}(x_n))_\varepsilon, x_n \in \mathcal{Z}_k \\
\beta_\xi &\text{otherwise}
\end{cases}
\end{align*}
it holds that
\begin{align*}
\Gamma\text{-} \lim_{\varepsilon \to 0}E_\varepsilon^{k,\delta,N}(u) =\Gamma\text{-} \lim_{\varepsilon \to 0} \frac{1}{4} \sum_{\xi \in V}\sum_{i,i+\xi \in \Omega_\varepsilon} \varepsilon^{d-1} c^{k,\delta,N,\varepsilon}_{i,\xi} (u_{\varepsilon i}-u_{\varepsilon(i+\xi)})^2 = E_{k,N,\delta}(u)
\end{align*}
with $E_{k,N,\delta}(u)$ of the form (\ref{last}). The statement follows by a diagonal argument for the convergence of the energies and the local volume fractions. 
\end{proof}

\end{document}